\documentclass{amsart}

\newtheorem{THM}{Theorem}
\newtheorem{Lemma}[THM]{Lemma}

\newtheorem{Cor}[THM]{Corollary}

\newcommand{\bR}{\mathbb{R}}
\newcommand{\be}{\begin{equation}}
\newcommand{\ee}{\end{equation}}
\newcommand{\mS}{\mathbb{S}}

\newcommand{\ip}[2]{\left\langle#1,#2\right\rangle}
\newcommand{\lp}{\left(}
\newcommand{\rp}{\right)}
\newcommand{\lb}{\left[}
\newcommand{\rb}{\right]}
\newcommand{\lc}{\left\{}
\newcommand{\rc}{\right\}}
\newcommand{\lab}{\left|}
\newcommand{\rab}{\right|}
\newcommand{\Lap}{\Delta}

\newcommand{\sD}{\mathcal{D}}

\newcommand{\E}{\mathbb{E}}
\newcommand{\Prob}{\mathbb{P}}

\newcommand{\vphi}{\varphi}

\newcommand{\QV}[1]{\left\langle#1\right\rangle}

\newcommand{\eps}{\varepsilon}
\newcommand{\ltwo}{\log_{(2)}}
\newcommand{\lthree}{\log_{(3)}}
\DeclareMathOperator{\spn}{span}

\newcommand\bigzero{\makebox(0,0){\text{\Huge0}}}

\begin{document}

\title[Martingales arising from minimal submanifolds, etc.]{Martingales arising from minimal
submanifolds and other geometric contexts}
\author{Robert W.\ Neel}
\address{Department of Mathematics, Lehigh University, Bethlehem, PA}
\begin{abstract} 
We consider a class of martingales on Cartan-Hadamard manifolds that includes Brownian motion on a minimal submanifold. We give sufficient conditions for such martingales to be transient, extending previous results on the transience of minimal submanifolds. We also give conditions for the almost sure convergence of the angular component (in polar coordinates) of a martingale in this class, including both  the negatively pinched case (using earlier results on martingales of bounded dilation), and the radially symmetric case with quadratic decay of the upper curvature bound. Applied to minimal submanifolds, this gives curvature conditions on the ambient Cartan-Hadamard manifold under which any minimal submanifold admits a non-constant, bounded, harmonic function. Though our discussion is primarily motivated by minimal submanifolds, this class of martingales includes diffusions naturally associated to ancient solutions of mean curvature flow and to certain sub-Riemannian structures, and we briefly discuss these contexts as well. Our techniques are elementary, consisting mainly of comparison geometry and Ito's rule.
\end{abstract}
\email{robert.neel@lehigh.edu}

\maketitle

\section{Introduction}

We study a class of degenerate martingales on Cartan-Hadamard manifolds, specifically their transience and angular behavior in large time. The results we obtain (as well as the methods of proof) are similar to those for Brownian motion on Cartan-Hadamard manifolds. The motivation is that such martingales arise in multiple geometric contexts, such as Brownian motion on minimal submanifolds, Brownian motion along a smooth mean curvature flow, and the canonical diffusion associated with certain sub-Riemannian geometries. This allows us to re-prove and extend geometric results, such as results on the transience of minimal submanifolds, and also provides a common perspective on what might otherwise seem like disparate objects in geometry, at least at the level of these relatively coarse properties.

Our discussion is motivated by the case of minimal submanifolds. In particular, let $M$ be a (smooth, complete) Cartan-Hadamard manifold of dimension $m$.  Markvorsen and Palmer \cite{MarkPalmGAFA} prove that if $N$ is a complete, $n$-dimensional, minimally immersed submanifold of M and either $n = 2$ and the sectional curvatures of M are bounded above by $-a^2 < 0$, or $n \geq 3$, then $N$ is transient. They prove this result by first deriving non-trivial capacity estimates. The more restrictive case when $n\geq 2$ and the sectional curvatures of $M$ are bounded above by $-a^2 < 0$ also follows from estimates on the first Dirichlet eigenvalue (or the fundamental tone) of submanifolds of a Cartan-Hadamard manifold, as discussed in more detail by Bessa and  F{\'a}bio \cite{BessaFabio}. In addition, we mention that Stroock (see Theorem 5.23 of \cite{StroockGeo}) proves that minimal submanifolds of dimension 3 or more of Euclidean space are transient (that is, the case $n\geq 3$ and $M=\bR^m$) using a simple stochastic argument which is very much in the spirit of the present work. (For general background on the role of transience in geometry and analysis, consult \cite{Grigoryan}.)

We wish to further this line of inquiry. Our first main result (concerning transience), applied to minimal submanifolds, yields the following.

\begin{Cor}\label{Cor:MinimalTrans}
Let $M$ be a Cartan-Hadamard manifold of dimension $m\geq 3$, and let $N$ be a properly immersed minimal submanifold of dimension $n$, for $2\leq n < m$. Then if either of the following two conditions hold:
\begin{enumerate}
\item $n=2$ and, in polar coordinates around some point, $M$ satisfies the curvature estimate
\[
K(r,\theta,\Sigma) \leq -\frac{1+2\eps}{r^2\log r} \quad \text{for $r>R$, and for all $\theta$ and $\Sigma\ni\partial_r$,}
\]
for some $\eps>0$ and $R>1$, or
\item $n\geq 3$, 
\end{enumerate}
we have that $N$ is transient.
\end{Cor}

Here $K(r,\theta,\Sigma)$ is the sectional curvature of a plane $\Sigma$ in $T_{(r,\theta)}M$ (in polar coordinates). Our choice of $2\eps$ rather than $\eps$ in the bound in the $n=2$ is simply for convenience.

Obviously, this improves the curvature bound in the case $n=2$. Indeed, the bound given above is sharp (in the sense that the theorem is not true for $\eps=0$), as discussed after Theorem \ref{THM:Trans}. Further, this result is a special case of a transience result for a broader class of martingales (given in Theorem \ref{THM:Trans}, of which the above is an immediate corollary) that we call rank-$n$ martingales and that also includes diffusions associated to ancient solutions of mean curvature flow and to certain sub-Riemannian structures (more on this in a moment). 

Our second main class of results consists of theorems giving conditions for the almost sure angular convergence of a rank-$n$ martingales. Using previous results on martingales of bounded dilation (more on this in a moment), we see that rank-$n$ martingales in Cartan-Hadamard manifolds of negatively pinched curvature have (random) angular limits. Applying this result to rank-$n$ martingales arising from minimal submanifolds gives

\begin{Cor}\label{Cor:UsingKendall}
Suppose that $M$ is Cartan-Hadamard manifold of dimension $m\geq 3$, and that the sectional curvatures of $M$ are bounded above and below by negative constants.  Let $N$ be an $n$-dimensional, properly immersed minimal submanifold (in $M$), for $n\geq 2$. Then $N$ admits a non-constant, bounded, harmonic function.
\end{Cor}

In order to explore the possibility of relaxing the upper curvature bound, we analyze the angular process in more detail, in the case when $n\geq 3$. We give complementary results on the finiteness or lack thereof of the quadratic variation of the martingale part of the angular process in Theorems \ref{THM:MartCon} and  \ref{THM:MartNoCon}. The role of the drift of the angular process is more subtle. Restricting our attention to radially symmetric Cartan-Hadamard manifolds, we show that the angular process converges even if some quadratic decay of the upper curvature bound is allowed (see Theorem \ref{THM:AngleCon}). Applying this to minimal submanifolds, we have

\begin{Cor}\label{Cor:MinHarm}
Suppose that $M$ is Cartan-Hadamard manifold of dimension $m\geq 4$, and that $M$ is radially symmetric around some point $p$.  Let $(r,\theta)$ be polar coordinates around $p$. Let $N$ be an $n$-dimensional, properly immersed minimal submanifold (in $M$), with $3\leq n<m$. Then if $M$ satisfies the curvature estimate
\[
-a^2 \leq
K(r,\theta,\Sigma) \leq -\frac{2+\eps}{r^2} \quad\text{when $r>R$, and for all $\theta$ and $\Sigma\ni\partial_r$,}
\]
for some $a>0$, $\eps>0$, and $R>1$, we have that $N$ admits a non-constant, bounded, harmonic function.
\end{Cor}

As mentioned, the properties of rank-$n$ martingales which we establish are motivated by minimal submanifolds, one consequence being that the applications in other geometric contexts, namely to mean curvature flow and certain sub-Rie\-man\-nian structures, appear less immediately compelling. For example, it's not clear how the transience of an associated diffusion along an ancient solution to mean curvature flow fits into the theory of such flows. In the sub-Riemannian case, the transience of the associated diffusion is a more natural property (akin to the transience of Brownian motion on a Riemannian manifold), but our discussion only applies to a rather restricted class of sub-Riemannian structures (ones equipped with an extension to a Riemannian metric which makes the diffusion generated by the sub-Laplacian a rank-$n$ martingale, as explained in Section \ref{Sect:SubR}). Nonetheless, we explain how these results apply in these two cases to the extent feasible, for a couple of reasons. We hope that this demonstrates the natural use of stochastic methods in these areas and indicates that further developments might be possible. It also highlights the common aspects of these geometric situations, in terms of basic properties accessible to stochastic analysis, and clarifies the role of various hypotheses. For instance, the fact that the proof of transience for a minimal submanifold applies unchanged to a class of sub-Riemannian structures shows that, in a sense, only the minimality of the submanifold is relevant, and not the fact that it's a submanifold. (Said differently, the integrability of the tangent spaces or lack thereof, in the sense of Frobenius, is irrelevant.)

From a technical point of view, not only are we able to treat a variety of geometric objects simultaneously, but our techniques are more or less completely elementary. On the geometric side, we use standard comparison results to pass from curvature estimates to estimates on Jacobi fields and their indices. In terms of stochastic analysis, Ito's rule (with the right choice of function to compose with our rank-$n$ martingale) and basic stopping time arguments are all that we use.

We have already alluded to a relationship between this work and previous work on Brownian motion on Cartan-Hadamard manifolds, such as the work of March \cite{March} or Hsu \cite{HsuBook}, which will become clear as we go. We now close this section with a further discussion of the connection with work on bounded-dilation martingales in geometry. Indeed, any rank-$n$ martingale with $n\geq 2$ is a special case of a martingale of bounded dilation (in fact, of 1-bounded dilation); see Definition 2.2.10 of Kendall's informative article \cite{KendallSurvey} and the surrounding discussion for the precise definition. Such martingales arise in the study of harmonic maps from one Riemannian manifold to another, as the image of a Brownian motion on the domain under the mapping. Various natural properties of such martingales have been studied, such as upper bounds on exit times from geodesic balls in Darling \cite{DarlingExits} and properties of the radial part in Kendall \cite{KendallRadial}. A central application of martingales of bounded dilation is to prove generalized Picard little theorems for harmonic maps (again, see Theorem 2.3.6 of \cite{KendallSurvey} and the references therein for an overview of such results). In fact, it was in the process of proving Picard little theorems that the result on angular convergence of bounded-dilation martingales on Cartan-Hadamard manifolds with negatively-pinched curvatures, as mentioned above, was established (see Theorem 2.3.5 of \cite{KendallSurvey}). Since rank-$n$ martingales are a narrower class of processes, we attempt to give more precise results than might be possible in more generality, such as the sharp curvature decay rate for transience of rank-2 martingales and the quadratic decay of the upper curvature bound for angular convergence of a rank-$n$ (with $n\geq 3$) martingale on a radially symmetric Cartan-Hadamard manifold. This perspective emphasizes the difference in the stochastic setting between previous work on bounded-dilation martingales and what we establish here for the more specialized situation of rank-$n$ martingales, but of course this reflects the fact that minimal submanifolds, and the other geometric situations we consider, are a different geometric problem than that of a harmonic map (even one of bounded dilation) from a given domain manifold into another manifold.

Finally, we mention that results showing that one can find a martingale on a manifold with a prescribed limit under some geometric constraints have been given by Darling \cite{DarlingExistence} and Arnaudon \cite{ArnaudonExistence}. These can be used to give a probabilistic construction of a solution to the harmonic mapping problem with given boundary values in certain cases. Unfortunately, such results do not apply in any obvious way to the construction of, say, a solution to the mean curvature flow, since that requires constructing a rank-$n$ martingale, for the appropriate value of $n$, which is a much more restricted class than that of just a martingale. Indeed, for an idea of how one might construct a solution to mean curvature flow via stochastic methods, see Soner and Touzi \cite{SonerTouzi}.

\section{Definitions and preliminary results}

\subsection{Rank-$n$ martingales}

Let $M$ be a Cartan-Hadamard manifold of dimension $m$; we will always assume that $M$ is smooth and complete. The main object of study is what we will call a rank-$n$ martingale, where $n\in\lc 1, \ldots, m-1 \rc$. This is a (continuous) process $X_t$ on $M$, possibly defined up to some explosion time $\zeta$, which (informally) is infinitesimally a Brownian motion on $\Lambda_t$, where $\Lambda_t$ is a (path) continuous, adapted choice of $n$-dimensional  subspace of $T_{X_t}M$. To be be more specific, we suppose that $X_t$ locally (in space and time) satisfies an SDE of the form
\begin{equation}\label{Eqn:RankN}
dX_t = \sum_{i=1}^n v_{i,t} \, dW^i_t
\end{equation}
where $(v_{1,t}, \ldots,  v_{n,t})$ is a continuous, adapted $n$-tuple of orthonormal vectors in $T_{X_t}M$. (Note that we're not claiming that a unique solution necessarily exists for any such choice of $v_{i,t}$, just that $X_t$ is a solution to such an equation.) Naturally, we think of the $v_{i,t}$ as a orthonormal basis for the $n$-dimensional subspace $\Lambda_t$ mentioned above. We generally think of $X_t$ as starting from a single point $X_0$, but this is not necessary (and it's occasionally useful to consider a more general initial distribution).

For clarity, consider a particularly nice example of such a process, Brownian motion on a (properly embedded) $n$-dimensional minimal submanifold. In that case, we can, on some local chart, let $v_{i,t} = v_i(X_t)$ for a smooth, orthonormal frame $v_1,\ldots,v_n$. Alternatively, if we want a smooth global (but non-Markov) choice of $v_{i,t}$, we can let them be given by the parallel transport of an orthonormal frame at $X_0$ along $X_t$, in the spirit of the Ells-Elworthy-Malliavin construction of Brownian motion on a manifold.

Returning to the general case of Equation \eqref{Eqn:RankN}, note that the differentials are Ito differentials, and thus $X_t$ is in fact an $M$-martingale (see \cite{HsuBook} for basic definitions in stochastic differential geometry). Also note that the law of $X_t$ should generally be thought of as depending on $\Lambda_t$, rather than on the specific choice of frame $v_{i,t}$.  In particular, let $\tilde{v}_{i,t}$ be another set of orthonormal frames for $\Lambda_t$. Then write $\tilde{v}_{i,t} = \sum_{j=1}^n c_{i,j,t} v_{j,t}$ for $i,j=1,\ldots, n$,  where $\lb c_{i,j}\rb_t$ is a continuous, adapted $O(n)$-valued process, and suppose that $\tilde{W}^i_t$ for $i=1,\ldots, n$ are independent Brownian motions satisfying the system
\[
\begin{bmatrix} d\tilde{W}_t^1  \\  \vdots \\ d\tilde{W}^n_t \end{bmatrix} = \lb c_{i,j}\rb_t
\begin{bmatrix} dW_t^1  \\  \vdots \\ dW^n_t \end{bmatrix} .
\]
Then we see that $X_t$ also satisfies the SDE
\[
dX_t = \sum_{i=1}^n \tilde{v}_{i,t} \, d\tilde{W}^i_t .
\]

Finally, we mention that our results extent in a natural way to the situation where our rank-$n$ martingale is (possibly) stopped at some stopping time prior to explosion. In particular, our results will apply on the set of paths that survive until explosion (which may mean they survive for all time, if the explosion time is infinite). We discuss this a bit further after the proof of Lemma \ref{BasicLemma} and at the end of Section \ref{Sect:MinSub}.  Note that this is a fairly natural situation. For example, if $N$ is a minimal submanifold-with-boundary, then Brownian motion on $N$ stopped at the boundary will be a rank-$n$ martingale stopped at the first hitting time of the boundary, as was used in \cite{MyJGA}. Nonetheless, to make the exposition cleaner, we will simply deal with rank-$n$ martingales as introduced above, aside from the two brief mentions just indicated.

\subsection{Comparison geometry}\label{Sect:Compare}

Again, let $M$ be a (smooth, complete) Cartan-Hadamard manifold of dimension $m$. For some point $p\in M$, let $(r,\theta)\in [0,\infty)\times \mS^{m-1}$ be polar coordinates; this gives a global coordinate system for $M$. Let $X_t$ be a rank-$n$ martingale on $M$, started at some point $X_0$.  We can write $X_t$ in coordinates as $(r_t,\theta_t)= \lp r(X_t),\theta(X_t)\rp$.

We wish to understand the behavior of $r_t$ in terms of the SDE it satisfies. We note that $r$ is smooth everywhere on $M$ except for $p$.  We'll see in a moment (in Lemma \ref{BasicLemma}) that if $X_t$ starts at $p$ it immediately leaves, and that, from anywhere else, the process almost surely never hits $p$. Thus we will assume that the process is not at $p$ (equivalently that that $r_t$ is positive), so that the behavior of $r$ at $p$ won't bother us.  (The situation is exactly analogous to what one sees for the the radial component of Brownian motion on Euclidean space.)
As usual, let $\partial_r=\nabla r$ be the unit radial vector field. Since we're free to rotate our orthonormal frame without changing the law of $X_t$, as discussed above, assume for convenience that $v_{2,t}, \ldots, v_{n,t}$ are perpendicular to $\partial_r$. Next, let $\vphi_t$ be the angle between $v_{1,t}$ and $\partial_r$, so that $\ip{v_{1,t}}{\partial_r}=\cos\vphi_t$. (While $\vphi_t$ is not uniquely determined, $\cos\vphi_t$ is, and we will see that it is only $\cos\vphi_t$ that matters in what follows. Nonetheless, we find it geometrically appealing to make reference to $\vphi_t$, and it causes no harm.)

Applying Ito's formula, we first see that that martingale part of $r_t$ evolves as $\cos\vphi_t \,d W_t$ for some Brownian motion $W_t$. (If we always choose our orthonormal frame as just discussed, then $W_t=W^1_t$, but in general we only care about the law of $r_t$, so this isn't necessary. More generally, $W_t$ is adapted to the filtration generated by $W^1_t,\ldots, W^n_t$, and $\vphi_t$ is the angle between $\Lambda_t$ and $\partial_r$.) In order to understand the bounded variation (or drift) term in the SDE, let $\gamma$ be the (unique) geodesic from $p$ to $X_t$. Then for $i=2,\ldots n$, let $J_i(s), s\in [0,r_t]$ be the (unique) Jacobi field along $\gamma$ with $J_i(0)=0$ and $J_i(r_t)=v_{i,t}$. For $i=1$ let $J_1(s), s\in [0,r_t]$ be the Jacobi field which is $0$ at $p$ and equal to the projection of $v_{1,t}$ onto the orthogonal complement of  $\partial_r$ (as a subspace of $T_{r_t}M$). In particular, $J_1(r_t)$ has length $\sin\vphi_t$ (assuming $\vphi_t$ is chosen to make this non-negative). Let $I(J_i)$ be the index of the Jacobi field $J_i$. Since the Hessian of $r$ is given in terms of these indices, we see that $r_t$ satisfies the SDE
\begin{equation}\label{Eqn:BasicDecomp}
dr_t = \cos\vphi_t \,d W_t + \frac{1}{2} \lp\sum_{i=1}^n I(J_i)\rp \, dt .
\end{equation}

We will deal with the indices $I(J_i)$ via standard comparison geometry (see Theorem 1.1 of \cite{SchoenYau} for the relevant version of the Hessian comparison theorem, and see Section 3.4 of \cite{HsuBook} for the application to Brownian motion). Since we will work with Cartan-Hadamard manifolds, we assume that all sectional curvatures are non-positive, and thus we take our comparison functions to be non-positive as well. In particular, let $\hat{K}(r)$ be a continuous, non-positive function on $[0,\infty)$ such that
\[
\hat{K}(r) \geq \max_{\theta, \Sigma\ni\partial_r} K(r,\theta,\Sigma) ,
\]
where the maximum is taken over all $\theta\in\mS^{m-1}$ and all two-planes $\Sigma$ in the tangent space at $(r,\theta)$ that contain $\partial_r$ (so that we deal with estimates on what are commonly called the radial curvatures). Further, let $\hat{G}(r)$ be the solution to the (scalar) Jacobi equation
\[
\hat{G}^{\prime\prime}(r) + \hat{K}(r) \hat{G}(r)=0, \quad \hat{G}(0)=0, \hat{G}^{\prime}(0)=1, \quad\text{on $r\in[0,\infty)$.}
\]
Then we have the comparison
\[
I(J_i) \geq |J_i(r_t)|^2\frac{\hat{G}^{\prime}(r_t)}{\hat{G}(r_t)} \quad \text{for each $i=1,\ldots,n$.}
\]
Similarly, 
let $\check{K}(r)$ be a continuous, non-positive function on $[0,\infty)$ such that
\[
\check{K}(r) \leq \min_{\theta, \Sigma\ni\partial_r} K(r,\theta,\Sigma) ,
\]
and let $\check{G}(r)$ be the solution to the (scalar) Jacobi equation
\[
\check{G}^{\prime\prime}(r) + \check{K}(r) \check{G}(r)=0, \quad \check{G}(0)=0, \check{G}^{\prime}(0)=1, \quad\text{on $r\in[0,\infty)$.}
\]
Then we have the comparison
\[
I(J_i) \leq |J_i(r_t)|^2\frac{\check{G}^{\prime}(r_t)}{\check{G}(r_t)} \quad \text{for each $i=1,\ldots,n$.}
\]
(We can think of $\hat{G}$ or $\check{G}$ as giving the analogous Jacobi fields of a radially symmetric comparison manifold.)
Observe that $|J_i(r_t)|^2 =1$ for $i=2,\ldots,n$ and $|J_1(r_t)|^2 =\sin^2\vphi_t$. Thus we have that (refer to Equation \eqref{Eqn:BasicDecomp})
\[\begin{split}
& \frac{1}{2} \lp\sum_{i=1}^n I(J_i)\rp \geq \frac{n-1+\sin^2\vphi_t}{2} \frac{\hat{G}^{\prime}(r_t)}{\hat{G}(r_t)} \\
 \text{or} \quad & \frac{1}{2} \lp\sum_{i=1}^n I(J_i)\rp \leq \frac{n-1+\sin^2\vphi_t}{2} \frac{\check{G}^{\prime}(r_t)}{\check{G}(r_t)} .
\end{split}\]
(Of course, we also have that $\hat{G}$ gives a lower bound, and $\check{G}$ an upper bound, on the lengths of the corresponding Jacobi fields, by the Rauch comparison theorem.)
Note that $\hat{G}$ and $\check{G}$ depend only on the geometry of $M$ (and $r$). Further, these expressions depend only on $\Lambda_t$ and not on the particular choice of the $v_{i,t}$ (so they are ``invariant'' in the geometric language sometimes used), which justifies the local (or even pointwise) nature of our choice of frame.

As mentioned, we will frequently denote $ \frac{1}{2} \lp\sum_{i=1}^n I(J_i)\rp$ by $v_t$. We now collect the specific estimates on the lengths of Jacobi fields (given in terms of $\hat{G}$ and $\check{G}$) and on $v_t$ that we will use. To make the notation less cumbersome, throughout the rest of the paper we will write $\ltwo r$ for $\log(\log r)$ and $\lthree r$ for $\log( \log(\log r ))$.

\subsection{Constant curvature estimates}\label{Sect:ConstEst}

First, because we assume that $M$ is a Cartan-Hadamard manifold, we can always take $\hat{K}\equiv 0$, for which we have
\[
\hat{G}(r)=r \quad\text{and}\quad v_t \geq \frac{n-1+\sin^2\vphi_t}{2r_t}
\]
(and with equality for $v_t$ if $K\equiv 0$, that is, if $M$ is in fact Euclidean space). This also shows that our comparison functions $\hat{G}$ and $\check{G}$ are always increasing and positive for positive $r$.

Next, suppose that, for some $R>0$ and some $a>0$,
\[
K(r,\theta,\Sigma) \geq -a^2 \quad\text{when $r>R$, and for all $\theta$ and $\Sigma\ni\partial_r$.} 
\]
Then we can take $\check{K}(r)$ to be non-positive and equal to $-a^2$ when $r\geq A$, for some $A>R$. We compute that
\[
\check{G}(r) = \frac{c_1}{a}\sinh(ar) + \frac{c_2}{a}\cosh(ar)
= \frac{1}{a}\lp c_2 + c_1\tanh(ar)\rp \cosh(ar) \quad\text{for $r\geq A$,}
\]
for some constants $c_1$ and $c_2$. If $c_3=c_1+c_2$, then $c_3>0$, using that $\check{G}(r)$ is positive for positive $r$. Further, an elementary computation shows that, for any $\delta>0$, there exists $B>A$ such that
\[
\frac{\check{G}^{\prime}(r)}{\check{G}(r)} \leq a(1+\delta) \quad\text{for $r \geq B$.}
\]

\subsection{A variable curvature estimate useful when $n=2$}\label{Sect:N2}

The following estimate, though stated in general, will be used in the case $n=2$.

Suppose that for some $\eps>0$ and some $R>1$, we have
\[
K(r,\theta,\Sigma) \leq -\frac{1+2\eps}{r^2\log r} \quad\text{when $r>R$, and for all $\theta$ and $\Sigma\ni\partial_r$.}
\]
(Here we use ``$2\eps$'' rather than ``$\eps$'' simply to have a more convenient constant in the following estimates.)  Then we can take $\hat{K}(r)$ to be non-positive and equal to 
\[
-\frac{1+\eps}{r^2\log r} \lp 1+\frac{\eps}{\log r}\rp \quad \text{for $r\geq A$,}
\]
for some $A>R$.
We see that
\[
G_1(r) = r\lp \log r\rp^{1+\eps} \quad\text{and}\quad
G_2(r) = G_1(r)  \int^r \frac{1}{s^2\lp\log s \rp^{2+2\eps}} \, ds 
\]
are a basis for the space of solutions to the Jacobi equation over $r\in[A,\infty)$. Thus
\[
\hat{G}(r) = c_1 G_1(r) + c_2 G_2(r) \quad\text{for $r\in[A,\infty)$,}
\]
for some constants $c_1$ and $c_2$. These constants are determined by the initial conditions at $r=A$. Nonetheless, $ \int^r 1/(s^2\lp\log s \rp^{2+2\eps}) \, ds$ is increasing and bounded, so let $\alpha\in(0,\infty)$ be its limit as $r\rightarrow \infty$. Then
\[
\hat{G}(r) \sim (c_1 +c_2\alpha) G_1(r) \quad\text{as $r\rightarrow\infty$,}
\]
where ``$\sim$'' means that the ratio of the two sides approaches 1.
We know that $\hat{G}(r)$ is positive and increasing for positive $r$, so if we let $c_3=c_1 +c_2\alpha$, then $c_3>0$. Explicit computation shows that $\hat{G}^{\prime}(r)$ is given by
\[\begin{split}
 & \lp c_1+c_2 \int^r \frac{1}{s^2\lp\log s \rp^{2+2\eps}} \, ds \rp
 (1+\eps)\lp\log r\rp^{\eps} \lp \frac{\log r}{1+\eps} +1 \rp +\frac{c_2}{r\lp\log r\rp^{1+\eps}} \\
&\sim c_3 \lb \lp\log r\rp^{1+\eps} + (1+\eps)\lp\log r\rp^{\eps} \rb \quad\text{as $r\rightarrow \infty$}.
\end{split}\]
Dividing this by $\hat{G}$ (and considering the large $r$ behavior) shows that for some $B>A$ and $c>0$, we have
\[
v_t \geq c\lp n-1+\sin^2\vphi\rp \lp \frac{1}{2r_t}+\frac{1+\eps}{2r_t \log r_t} \rp
\quad\text{for $r_t\geq B$}.
\]

\subsection{Variable curvature estimates useful when $n\geq 3$}\label{Sect:NBig}

The previous estimate will be useful to us in the $n=2$ case. We now develop similar estimates for use in the $n\geq 3$ case.

Suppose that for some $\eps>0$ and some $R>1$, we have
\[
K(r,\theta,\Sigma) \leq -\frac{\frac{1}{2}+\eps}{r^2\log r} \quad\text{when $r>R$, and for all $\theta$ and $\Sigma\ni\partial_r$.}
\]
We may as well assume that $\eps<1/2$, since if the above holds for some $\eps$, it holds for any smaller $\eps$.
Then we can take $\hat{K}(r)$ to be non-positive and equal to 
\[
-\frac{\frac{1}{2}+\eps}{r^2\log r} \lp 1-\frac{\frac{1}{2}-\eps}{\log r}\rp \quad \text{for $r> R$.}
\]
We see that
\[
G_1(r) = r\lp \log r\rp^{\frac{1}{2}+\eps} \quad\text{and}\quad
G_2(r) = G_1(r) \int^r \frac{1}{s^2\lp\log s \rp^{1+2\eps}} \, ds 
\]
are a basis for the space of solutions to the Jacobi equation over $r\in(R,\infty)$. As before, $ \int^r 1/(s^2\lp\log s \rp^{1+2\eps}) \, ds$ is increasing and bounded, so let $\alpha\in(0,\infty)$ be its limit. Then
\[
\hat{G}(r) \sim (c_1 +c_2\alpha) G_1(r) \quad\text{as $r\rightarrow\infty$.}
\]
We know that $\hat{G}(r)$ is positive and increasing for positive $r$, so if we let $c_3=c_1 +c_2\alpha$, then $c_3>0$. Explicit computation analogous to the above gives
\[
\hat{G}^{\prime}(r) \sim c_3 \lb \lp\log r\rp^{\frac{1}{2}+\eps} + \lp\frac{1}{2}+\eps\rp\frac{1}{\lp\log r\rp^{\frac{1}{2}-\eps}} \rb \quad\text{as $r\rightarrow \infty$}.
\]
Dividing this by $\hat{G}$ (and considering the large $r$ behavior) shows that for some $B>R$, we have
\[
v_t > \frac{\frac{3}{4}\lp n-1+\sin^2\vphi\rp}{2r_t}\lp 1+\frac{\frac{1}{2}+\eps}{\log r} \rp
\quad\text{for $r_t\geq B$}.
\]
(The $3/4$ could be replaced by any positive real less than 1 by changing $B$, but it's less hassle for us to just pick an explicit coefficient here.)

In a complementary direction, assume that for some $R>1$, we have
\[
K(r,\theta,\Sigma) \geq -\frac{1/2}{r^2\log r} \quad\text{when $r>R$, and for all $\theta$ and $\Sigma\ni\partial_r$.}
\]
Then we can take $\check{K}(r)$ to be equal to 
\[
-\frac{1/2}{r^2\log r} \lb 1+ \frac{1}{\ltwo r}-\frac{1}{2\log r}-\frac{1}{2\log r \ltwo r } \rb
\quad \text{for $r> A$,}
\]
for some $A>R$.

We see that
\[
G_1(r) = r\lp \log r \ltwo r \rp^{\frac{1}{2}} 
\quad\text{and}\quad
G_2(r) =  G_1(r)  \int^r \frac{1}{s^2 \log s \ltwo s } \, ds
\]
are a basis for the space of solutions to the Jacobi equation over $r\in(A,\infty)$. Just as above, there is $c_3>0$ such that
\[
\check{G}(r)= c_1G_1(r)+c_2G_2(r) \sim c_3 G_1(r) \quad\text{as $r\rightarrow\infty$.}
\]
Another explicit computation gives
\[
\check{G}^{\prime}(r) \sim c_3 \lp \log r \rp^{\frac{1}{2}} \lp \ltwo r \rp^{\frac{1}{2}}
\lb 1+ \frac{1/2}{\log r} +\frac{1/2}{\log r \ltwo r } \rb
\quad\text{as $r\rightarrow \infty$}.
\] 
Dividing this by $\check{G}$ (and considering the large $r$ behavior) shows that for some $B>A$, we have
\[
v_t \leq  \frac{5}{4}  \frac{n-1+\sin^2\vphi_t}{2r_t}\lb 1+\frac{1/2}{\log r_t}+\frac{1/2}{\log r_t \ltwo r_t } \rb
\quad\text{for $r_t\geq B$.}
\]
(Again, the $5/4$ is chosen for convenience; any positive real greater than 1 could be used by changing $B$.)

Finally, suppose that for some $\eps>0$ and some $R>0$, we have
\[
K(r,\theta,\Sigma) \leq -\frac{2+\eps}{r^2} \quad\text{when $r>R$, and for all $\theta$ and $\Sigma\ni\partial_r$.}
\]
Then for $\delta>0$ such that $2+\eps=(2+\delta)(1+\delta)$, we can take $\hat{K}(r)$ to be non-positive and equal to
\[
-\frac{(2+\delta)(1+\delta)}{r^2} \quad\text{for $r>R$.}
\]
We find that
\[
G_1(r) = r^{2+\delta} 
\quad\text{and}\quad
G_2(r) = \int^r \frac{1}{s^{4+2\delta}}  \, ds G_1(r) 
\]
are a basis for the space of solutions to the Jacobi equation over $r\in(R,\infty)$. As before, we see that there is $c_3>0$ such that $\hat{G}(r)\sim c_3G_1(r)$, and we compute
\[
\hat{G}^{\prime}(r)\sim  c_3 (2+\delta)r^{1+\delta} .
\]
It follows that for some $B>R$ and $c>0$, we have
\[
v_t > c\frac{n-1+\sin^2\phi}{r_t} \quad\text{for $r_t>B$.}
\]

\subsection{Basic properties}

We now give a lemma showing that rank-$n$ martingales share several basic properties of Brownian motion.
\begin{Lemma}\label{BasicLemma}
Let $M$ be a Cartan-Hadamard manifold of dimension $m$, and let $X_t$ be a rank $n$ martingale on $M$ (for $2\leq n < m$), started from any initial point $X_0$. For any $p\in M$, let $(r,\theta)$ be polar coordinates around $p$. Then
\begin{enumerate}
\item $\limsup_{t\rightarrow \zeta} r_t =\infty$, almost surely, and
\item $r_t>0$ for all $t\in(0,\zeta)$, almost surely.
\end{enumerate}
\end{Lemma}
\begin{proof}
Note that the sectional curvatures of $M$ are all non-positive, and thus we have that
\[
dr_t = \cos\vphi_t \,d W_t + v_t \, dt
\]
where, as just discussed, $v_t=\sum_{i=1}^n I(J_i)$ is time-continuous (and adapted, of course) and satisfies 
\begin{equation}\label{Eqn:v_t}
v_t\geq \frac{n-1+\sin^2\vphi_t}{2r_t} .
\end{equation}
Let $\sigma_x$ be the first hitting time of $\{r=x\}$, for $x\geq 0$.
Ito's rule gives that
\[
d\lp r^2 \rp_t = 2 r_t \cos\vphi_t \,d W_t  +\lp 2 r_t v_t +\cos^2\vphi_t\rp \, dt .
\]

Equation \eqref{Eqn:v_t} implies $2 r_t v_t +\cos^2\vphi_t \geq n$, and thus, for $C>r_0$,
\[
\E\lb r^2_{\sigma_C\wedge t} \rb \geq r^2_0 + n\E\lb \sigma_C \wedge t \rb .
\]
Using that $\E\lb r^2_{\sigma_C\wedge t} \rb\leq C^2$, dominated convergence lets us take $t\rightarrow \zeta$, and since $\sigma_C\leq \zeta$ we see that
\[
\E\lb \sigma_C \rb \leq \frac{1}{n}\lp C^2-r^2_0 \rp .
\]
In particular, $\Prob\lp \sigma_C <\infty\rp =1$. Because $\sigma_C=\zeta$ can only happen if both are infinite by path continuity, we also have $\Prob\lp \sigma_C <\zeta\rp =1$. Since this holds for all $C>r_0$, and since $r_t$ has continuous paths, we conclude that $\limsup_{t\rightarrow \zeta} r_t =\infty$, almost surely.

The proof of the second part mimics that of Proposition 3.22 of \cite{KS}. It is immediate from Equation \eqref{Eqn:RankN} (say, by using normal coordinates around $p$) that if $X_0=p$, the process immediately leaves $p$, almost surely, which is equivalent to $r_t$ immediately becoming positive, almost surely. Thus it is enough to prove the result under the assumption that $r_0>0$, and we now assume this.

From the first part and the definition of explosion, we see that
\begin{equation}\label{Eqn:Lem1First}
\Prob\lp \text{$\sigma_k<\zeta$ for all integers $k>r_0$, and $\lim_{k\rightarrow\infty}\sigma_k =\zeta$}\rp =1.
\end{equation}
Ito's rule shows that
\begin{equation}\label{Eqn:LogDecomp}
d\lp\log r\rp_{t} = \frac{1}{r_t}\cos\vphi_t \, dW_t +\lp \frac{1}{r_t}v_t - \frac{1}{2r^2}\cos^2\vphi_t\rp \, dt ,
\end{equation}
at least when $r>0$, which is all we will need.
Equation \eqref{Eqn:v_t} implies that the coefficient of $dt$ is greater than or equal to
\[
\frac{n-1+\sin^2\vphi_t -\cos^2\vphi_t}{2r^2_t} ,
\]
and $n\geq 2$ means that this is always non-negative. So $\log r_t$ is a (local) sub-martingale.

Thus, if $k$ is an integer such that $(1/k)^k<r_0<k $ (which will be true for all sufficiently large $k$), by the first part of the lemma, we know that $\Prob\lp \sigma_{(1/k)^k}\wedge\sigma_k<\infty \rp=1$. So dominated convergence and the fact that $\log r_t$ is a (local) sub-martingale give
\[\begin{split}
\log r_0 &\leq \E\lb\log r_{\sigma_{(1/k)^k}\wedge\sigma_k}\rb \\
& = -k\log k \Prob\lp \sigma_{(1/k)^k}
\leq \sigma_k \rp + \log k \lp 1- \Prob\lp \sigma_{(1/k)^k}\leq \sigma_k \rp \rp .
\end{split}\]
Algebra yields
\[
\Prob\lp \sigma_{(1/k)^k}\leq \sigma_k \rp \leq \frac{\log k-\log r_0}{(k+1)\log k} ,
\]
and letting $k\rightarrow\infty$ shows that
\begin{equation}\label{Eqn:Lem1Second}
\lim _{k\rightarrow\infty} \Prob\lp \sigma_{(1/k)^k}\leq \sigma_k \rp =0.
\end{equation}

Now $\sigma_0$ is the first hitting time of $\{r_t=0\}$, and we see that $\sigma_0\leq \sigma_{(1/k)^k}$ for all $k$. Then Equations \eqref{Eqn:Lem1First} and \eqref{Eqn:Lem1Second} imply that
\[
\Prob \lp \sigma_0<\zeta \rp = \lim_{k\rightarrow\infty} \Prob \lp \sigma_0<\sigma_k \rp \leq 
\lim _{k\rightarrow\infty} \Prob\lp \sigma_{(1/k)^k}\leq \sigma_k \rp =0.
\]
This is equivalent to the desired result, namely that $r_t>0$ for all $t\in(0,\zeta)$, almost surely.
\end{proof}

The first part of the lemma says that a rank-$n$ martingale almost surely leaves every compact set. We will routinely use this, in much the way that we did in the second part of the proof where it implied that $\log r_t$ left any interval of the form $((1/k)^k,k)$ (prior to $\zeta$). Since $p$ was arbitrary, the second part means that, like Brownian motion, a rank-$n$ martingale does not charge points. It also justifies our assertion from the introduction that, since are interested in the long-time behavior of our rank-$n$ martingales, we need not worry about the singularity of our polar coordinates at $p$, because the process will avoid $p$ at all positive times almost surely.

We also note that, in light of the above lemma and its proof, our earlier comment about allowing our rank-$n$ martingales to be stopped prior to (possible) explosion becomes clearer. In this case, if $\eta$ is such a stopping time, then (for example) the first part of the lemma becomes the statement that $\limsup_{t\rightarrow \zeta} r_t =\infty$ almost surely on the set of paths with $\eta=\infty$. The proof is a straightforward modification of the above.

\section{Transience of rank-$n$ martingales}

Our goal here is to determine conditions for $X_t$ to be transient,  that is, conditions such that $\lim_{t\rightarrow \zeta} r(X_t)=\infty$ almost surely.

\begin{THM}\label{THM:Trans}
Let $M$ be a Cartan-Hadamard manifold of dimension $m$, and let $X_t$ be a rank-$n$ martingale on $M$ (for $2\leq n < m$). Then if either of the following two conditions hold:
\begin{enumerate}
\item $n=2$ and, in polar coordinates around some point, $M$ satisfies the curvature estimate
\[
K(r,\theta,\Sigma) \leq -\frac{1+2\eps}{r^2\log r} \quad \text{for $r>R$, and for all $\theta$ and $\Sigma\ni\partial_r$,}
\]
for some $\eps>0$ and $R>1$, or
\item $n\geq 3$, 
\end{enumerate}
we have that $X_t$ is transient.
\end{THM}

\begin{proof}
Note that the sectional curvatures of $M$ are all non-positive, and thus in either case we have that
\[
dr_t = \cos\vphi_t \,d W_t + v_t \, dt
\]
where $v_t=\sum_{i=1}^n I(J_i)$ is time-continuous (and adapted, of course) and satisfies 
\[
v_t\geq \frac{n-1+\sin^2\vphi_t}{2r_t} .
\]

For $n\geq 3$, this is enough. In particular, in this case we have that
\begin{equation}\label{Eqn:TransEstimate}
v_t\geq \frac{2+\sin^2\vphi_t}{2r_t} \geq \frac{1}{r_t}.
\end{equation}
(The intuitive point is that the drift is at least as large as for a 3-dimensional Bessel process, while the quadratic variation of the martingale part grows no faster than for a 3-dimensional Bessel process, and thus one expects $r_t$ to be ``at least as transient'' as a 3-dimensional Bessel process.) Ito's rule followed by an application of Inequality \eqref{Eqn:TransEstimate} and algebra gives
\[\begin{split}
& d\lp \frac{1}{r}\rp_t = -\frac{1}{r^2_t}\cos\vphi_t \, dW_t +\lp \frac{1}{r^3_t}\cos^2\vphi_t -\frac{v_t}{r^2_t}\rp \, dt \\
\text{where}\quad& \frac{1}{r^3_t}\cos^2\vphi_t -\frac{v_t}{r^2_t} \leq -\frac{\sin^2\vphi_t}{r^3_t}
\leq 0 .
\end{split}\]
In particular, $1/r_t$ is a (local) supermartingale, at least for $r>0$.

We know that $\sigma_k$ is finite for all integers $k>r_0$, almost surely. Choose any $a>0$. Then the event $\liminf_{t\rightarrow \zeta} r_t \leq a$ coincides, up to a set of probability zero, with the event that $r_t$ hits the level $a$ after $\sigma_k$ for all $k$ that are also larger than $a$. Choose such a $k$ and a $b>k$. Let $\tilde{\sigma}_a$ be the first hitting time of $a$ (for $r_t$) after $\sigma_k$, and similarly for $\tilde{\sigma}_b$. We know that $\tilde{\sigma}_b$ is almost surely finite. This, along with the fact that $1/r_t$ is a (local) supermartingale and our choice of stopping times implies
\[
\frac{1}{k} = \E\lb \frac{1}{r_{\sigma_k}}\rb \geq \E\lb \frac{1}{r_{\tilde{\sigma}_a\wedge\tilde{\sigma}_b}}\rb = \frac{1}{a} \Prob\lp \tilde{\sigma}_a<\tilde{\sigma}_b \rp +  \frac{1}{b}\lp1- \Prob\lp \tilde{\sigma}_a<\tilde{\sigma}_b \rp\rp .
\]
This, in turn, yields
\[
 \Prob\lp \tilde{\sigma}_a<\tilde{\sigma}_b \rp \leq \frac{\frac{1}{k}-\frac{1}{b}}{\frac{1}{a}-\frac{1}{b}} .
\]
If $\tilde{\sigma}_a$ is finite, it must be less than $\tilde{\sigma}_b$ for all sufficiently large $b$ (up to a set of probability zero), so letting $b\rightarrow \infty$ shows that
\[
\Prob\lp \tilde{\sigma}_a<\infty \rp \leq \frac{a}{k} .
\]
Because the right-hand side of the above goes to zero as $k\rightarrow \infty$, we see that the probability of $r_t$ returning to the level $a$ after every $\sigma_k$ (for sufficiently large $k$) is zero. It follows that $\liminf_{t\rightarrow \zeta} r_t > a$ almost surely. Since $a$ was arbitrary, we conclude that $\liminf_{t\rightarrow \zeta} r_t=\infty$ almost surely. Then $\lim_{t\rightarrow \zeta} r_t=\infty$ almost surely, which is equivalent to the transience of $X_t$.

The first part (the $n=2$ case) is similar, the difference being that we must replace $1/r$ with a more suitable function. Using the curvature bound and the results of Section \ref{Sect:N2}, we see that, for some $B>1$, we have
\[
v_t \geq \lp 1+\sin^2\vphi_t\rp \lp \frac{1}{2r_t} + \frac{1+\eps}{2r_t \log r_t} \rp
\quad \text{for $r_t>B$.}
\]
Now Ito's rule gives (for $r_t>1$)
\[\begin{split}
d\lp \frac{1}{\lp\log r\rp^{\eps}}\rp_t &= \frac{-\eps \cos\vphi_t}{r \lp\log r_t\rp^{(1+\eps)}} \, dW_t \\
& + \frac{\eps}{r \lp\log r_t\rp^{(1+\eps)}} \lb \frac{\cos^2\vphi_t}{2r_t}\lp \frac{1+\eps}{\log r_t}+1\rp -v_t \rb \, dt.
\end{split}\]
Combing this with the upper bound for $v_t$, we see that, for $r_t>B$, the coefficient of $dt$ (in other words, the infinitesimal drift) is less than or equal to
\[
\frac{\eps}{2r^2 \lp\log r_t\rp^{(1+\eps)}} \lp 1+ \frac{1+\eps}{\log r_t}\rp \lp \cos^2\vphi_t -\sin^2\vphi_t -1\rp \leq 0.
\]
It follows that $1/(\log r_t)^{\eps}$ is a (local) supermartingale for $r_t>A$.

If we now choose $k$, $a$, and $b$ (and the corresponding notation) as above, with the additional stipulation that $a>B$, similar logic gives
\[\begin{split}
& \frac{1}{\lp\log k\rp^{1+\eps}} = \E\lb\frac{1}{\lp\log r_{\sigma_k}\rp^{1+\eps}} \rb
\geq \E\lb\frac{1}{\lp\log r_{\tilde{\sigma}_a\wedge\tilde{\sigma}_b}\rp^{1+\eps}} \rb \\
& \quad = \frac{1}{\lp\log a\rp^{1+\eps}}  \Prob\lp \tilde{\sigma}_a<\tilde{\sigma}_b \rp +
\frac{1}{\lp\log b\rp^{1+\eps}} \lp 1- \Prob\lp \tilde{\sigma}_a<\tilde{\sigma}_b \rp\rp .
\end{split}\]
It follows that
\[
\Prob\lp \tilde{\sigma}_a<\tilde{\sigma}_b \rp \leq \frac{\frac{1}{\lp\log k\rp^{1+\eps}} -\frac{1}{\lp\log b\rp^{1+\eps}} }{\frac{1}{\lp\log a\rp^{1+\eps}}  -\frac{1}{\lp\log b\rp^{1+\eps}} } ,
\]
and letting $b\rightarrow \infty$ shows that $\Prob\lp \tilde{\sigma}_a<\infty \rp \leq (\log a/\log k)^{1+\eps}$. Since this last quantity goes to zero as $k\rightarrow \infty$, just as before we conclude that $\liminf_{t\rightarrow \zeta} r_t > a$ almost surely. Because this holds for any $a>B$, it follows that $\lim_{t\rightarrow \zeta} r_t=\infty$ almost surely, which is equivalent to the transience of $X_t$.\end{proof}

Note that the estimates used in the proof of the $n\geq 3$ case are not sharp. This indicates that some positive curvature could be allowed in this case. This, however, would take us outside the context of Cartan-Hadamard manifolds and thus require additional topological assumptions in order to have global polar coordinates, and for this reason we prefer to restrict our attention to Cartan-Hadamard manifolds.

In the $n=2$ case, however, the above is sharp, in the sense that the result does not hold for $\eps=0$. This follows from known results for transience and recurrence of Brownian motion on surfaces, once we observe that a radially symmetric surface can be realized as a totally geodesic submanifold of a radially symmetric 3-manifold, so that Brownian motion on a radially symmetric surface is included as a special case of the above (see also Section \ref{Sect:VanishingDrift}).

\section{Transience in geometric contexts}

We now establish the connection between rank-$n$ martingales and various geometric objects. 

\subsection{Minimal submanifolds}\label{Sect:MinSub}
First, let $N$ be an $n$-dimensional, properly immersed, minimal submanifold of $M$. Let $X_t$ be Brownian motion on $N$, viewed as a process in $M$ (under the immersion, of course). Then $X_t$ is a rank-$n$ martingale in $M$, as mentioned in the introduction.

To see this, let $\tilde{v}_{i,t}\subset TN$ be such that the solution to
\[
d\tilde{X}_t = \sum_{i=1}^n \tilde{v}_{i,t} dW^i_t
\]
is Brownian motion on $N$. (Locally, we can just let $\tilde{v}_{i,t} = \tilde{v}_i(\tilde{X}_t)$ for a smooth orthonormal frame $\tilde{v}_1,\ldots,\tilde{v}_n$. Alternatively, we can let the $\tilde{v}_{i,t} $ be given by the parallel transport of an orthonormal frame at $\tilde{X}_0$ along $\tilde{X}_t$, in the spirit of the Ells-Elworthy-Malliavin construction of Brownian motion on a manifold.) Now let $X_t$ and $v_{i,t}$ be the images of $\tilde{X}_t$ and $\tilde{v}_{i,t}$ under the immersion. Then in general, Ito's rule shows that they satisfy the SDE
\[
dX_t = \sum_{i=1}^n v_{i,t} dW^i_t -\frac{1}{2}H(X_t)\, dt ,
\]
where $H$ is the mean curvature vector of $N$ (as a submanifold of $M$). Here, we normalize $H$ so that, if $y_1,\ldots,y_m$ are normal coordinates for $M$ centered at a point $p$ in (the image of) $N$, then
\[
H(p) = -\sum_{i=1}^m \lp \Lap_N y_i|_N\rp \partial_{y_i} .
\]
Since we're assuming that $N$ is minimal (and, of course, the $v_{i,t}$ are orthonormal since the immersion is isometric), $H\equiv 0$, and we see that $X_t$ is a rank-$n$ martingale.

Properness of $N$ implies that $X_t$ explodes relative to the topology of $N$ if and only if it explodes relative to the topology of $M$. (In particular, $\zeta$ is independent of whether we view $X_t$ as a process on $N$ or on $M$.) Thus Theorem \ref{THM:Trans}, applied to this case, implies Corollary \ref{Cor:MinimalTrans}. Note that the assumption of properness can be dropped, via a natural application of allowing a rank-$n$ martingale to be stopped prior to explosion. If $N$ is not properly immersed, then Brownian motion on $N$ might explode (relative to the topology of $N$) without also exploding in $M$ (relative to the topology of $M$). Let $\eta$ be the explosion time relative to the topology of $N$. Clearly $\eta$ is a stopping time, and we consider $X_t$ run until $\eta$. On the set of paths where $\eta=\zeta$, Theorem \ref{THM:Trans} (or more precisely, a simple modification of its proof) shows that $X_t$ is transient (noting that it's still true that transience on $M$ implies transience on $N$). On the set of paths where $\eta<\zeta$ (which is the only other possibility), we have in particular that $\eta<\infty$. Thus $X_t$ is almost surely transient on this set as well, and the claim that the assumption of properness can be dropped follows.

\subsection{Mean curvature flow}
Next, we note that this idea generalizes to mean curvature flow. To describe this, for a manifold $N$, with a smooth structure, of dimension $n$ (for $2\leq n < m$), let
\[
f_{\tau}(y):(-\infty,0]\times N\rightarrow M
\]
be an ancient solution to mean curvature flow. (The term ``ancient'' refers to the fact that the solution is defined at all past times. Solutions defined for $\tau\in[0,\infty)$ are called immortal, and solutions defined for all $\tau$, eternal.) In particular, let $g_{\tau}$ be the (smoothly-varying) metric induced on $N$ at time $\tau$ by the immersion, let $\Lap_{\tau}$ be the associate Laplacian, and let $H_{\tau}$ be the associated mean curvature vector (with the same normalization as above, for any fixed time) for $f_{\tau}(N)$. Then $f_{\tau}$ is a smooth function, proper as a map from $N$ to $M$ at every fixed time $\tau$, satisfying the differential equation
\[
\partial_{\tau} f_{\tau}(y) = -\frac{1}{2} H_{\tau}(y) \quad\text{for all $\tau\in[0,\infty)$ and $y\in N$.}
\]
Note that the factor of $1/2$ in front of $H_{\tau}$ is non-standard. This is the same difficulty as encountered in normalizing the heat equation; the factor of $1/2$ is better suited to stochastic analysis, but analysts and geometers prefer not to include it. Rescaling time allows one to recover the standard normalization. Since our results are all for the asymptotic behavior of our process, though, rescaling time doesn't change them, and thus the unusual normalization causes no particular trouble.

Note that we don't require the solution to develop a singularity at some positive time (so that eternal solutions are a special case of ancient solutions). Further, we assume that the solution is smooth even at time 0 (which is the meaning of our assumption $\tau\in[0,\infty)$). In other words, if there is a singularity (such as collapse to a ``round point''), we don't start our rank-$n$ martingale from there. This is so that we can make our choice of $\Lambda_t$ continuous, in keeping with our desire to avoid technical difficulties here.

Now suppose that $\tilde{X}_t$ satisfies
\[
d\tilde{X}_t = \sum_{i=1}^n \tilde{v}_{i,t} dW^i_t ,
\]
where the $\tilde{v}_{i,t}$ are an orthonormal frame at $\tilde{X}_t$ with respect to the metric $g_{-t}$. We can always find such a process; for example, by letting the $\tilde{v}_{i,t}$ come from a local (in both space and time) smooth choice of time-varying orthonormal frame. We think of $\tilde{X}_t$ as being Brownian motion along the mean curvature flow, run backwards in time. Indeed, note that $\tilde{X}_t$ is the (inhomogeneous) diffusion associated to $\frac{1}{2}\Lap_{-t}$.

If we again let $X_t$ be the image of $\tilde{X}_t$ under $f_{-t}$, then $X_t$ is a rank-$n$ martingale (on $M$). To see this, let $v_{i,t}$ be the pushforward of $\tilde{v}_{i,t}$, and note that the $v_{i,t}$ are orthonormal since $f_{-t}$ is an isometric immersion for all $t$. Then Ito's formula gives
\[
dX_t = \sum_{i=1}^n v_{i,t} dW^i_t -\frac{1}{2}H_{-t}(X_t)\, dt - \left.\frac{\partial f}{\partial \tau}\right|_{\tau=-t}\lp \tilde{X}_t\rp \, dt .
\]
But using that $f_{\tau}$ is a solution to the mean curvature flow, this reduces to
\[
dX_t = \sum_{i=1}^n v_{i,t} dW^i_t
\]
as desired.  (The minimal submanifold situation above is just the special case of a constant solution to the mean curvature flow.) This also explains why we consider our ``inhomogeneous Brownian motion'' to be run backwards in time with respect to the flow. This phenomenon (namely, ``process time'' running in the opposite direction from ``PDE time'') is familiar, even arising in the standard approach to the heat equation on the real line via Brownian motion.

Of course, for a non-ancient solution to the mean curvature flow, the same procedure leads to a rank-$n$ martingale run for a finite time. However, since all of our results in the current paper concern the asymptotic behavior of rank-$n$ martingales, we don't consider this case.

A geometric interpretation of the transience of this process is less obvious than it was for minimal submanifolds. Certainly, though, it contains some information about the flow, such as the (fairly basic) fact that an ancient solution cannot be contained in a compact subset of $M$ (for all time). For a refinement of this idea, we mention the following. In \cite{MyJGA}, rank-2 martingales in $\bR^3=\lc(x_1, x_2, x_3):x_i\in\bR\rc$ were studied, with an eye toward classical minimal surfaces. Let $r=\sqrt{x_1^2+x_2^2}$. It was proved that for any $c>0$, there is a positive integer $L$ such that any rank-2 martingale (as considered in the present paper) exits the region
\[
A=\lc r>e^L\text{ and } |x_3|<\frac{cr}{\sqrt{\log r \ltwo r}} \rc
\]
in finite time, almost surely (this is a restatement of Theorem 2 of \cite{MyJGA}). Thus, we see that an ancient solution to mean curvature flow, for surfaces in $\bR^3$, cannot be contained in $A$ for all $\tau\in(-\infty,0]$.

Finally, we (again) mention that the relationship between rank-$n$ martingales and mean curvature flow can be used to represent mean curvature flow in terms of a type of stochastic target problem. This is done, in the case when the ambient space is $\bR^n$, in \cite{SonerTouzi}.

\subsection{Sub-Riemannian geometry}\label{Sect:SubR}

Our final example of rank-$n$ martingales arising in geometry is as follows. Again, starting with a Cartan-Hadamard manifold $M$ (of dimension $m\geq 3$), choose a smooth rank-$n$ distribution satisfying the bracket-generating property. That is, let $\sD$ be a smooth map which assigns an $n$-dimensional subspace $\sD_y$ of $T_yM$ to each $y\in M$. Further, if
\[
\sD_y^k = \spn\lc [w_1,[\ldots[w_{k-1},w_k]]]_y : w_i(z)\in \sD_z \text{ }\forall z\in M \text{ and $w_i$ is smooth}\rc ,
\]
we assume that for each $y\in M$, there exists an integer $k(y)\geq 2$ such that $T_yM=\sD_y^{k(y)}$ (this is the bracket-generating property). Each subspace $\sD_y$ can be given a Riemannian metric by restricting the metric on $M$ to $\sD_y$. This gives a sub-Riemannian structure on $M$. (See \cite{Montgomery} for background on sub-Riemannian geometry.)

Further, suppose we choose a smooth volume form on $M$ (in general, there is no canonical volume associated to a sub-Riemannian structure, although intrinsic choices, such as the Popp volume, can  be considered). Then let $\Lap_{s}$ be the associated sub-Laplacian (defined as the divergence of the horizontal gradient). In general, if $v_1,\ldots v_n$ is a local orthonormal frame for $\sD$, then the sub-Laplacian will be locally given by $\sum v_i^2$ plus a first-order term.

We are interested in the case when the sub-Laplacian turns out to be a rank-$n$ martingale with respect to the ambient Riemannian metric. (In our earlier notation, we will then have $\Lambda_t = \sD_{X_t}$.) This will happen if, at any point, we can find normal coordinates such that the sub-Laplacian can be written as a sum of squares of coordinate vector fields at that point. (This is an Ito-type sum of squares condition, as opposed to the more common Stratonovich-type sum of squares condition, in which one assumes that the the sub-Laplacian is a sum of squares of smooth vector fields, such as the $v_i$ as above, on an open set.) While this will certainly not be true in general, one can consider Carnot groups or unimodular Lie groups (as discussed in \cite{Ugo}). 

At any rate, if the diffusion associated to $\frac{1}{2}\Lap_s$ is a rank-$n$ martingale (on $M$ with the original Riemannian metric), then by assumption our results apply in this situation.

\begin{Cor}
Let $M$ be a Cartan-Hadamard manifold of dimension $m\geq 3$. With $2\leq n < m$, let the rank-$n$ distribution $\sD$ (as above) with the restriction metric be a sub-Riemannian structure on $M$, and suppose this sub-Riemannian structure is given a volume form such that the associated sub-Laplacian $\Lap_s$ gives rise to a rank-$n$ martingale. Then if either of the following two conditions hold:
\begin{enumerate}
\item $n=2$ and, in polar coordinates around some point, $M$ satisfies the curvature estimate
\[
K(r,\theta,\Sigma) \leq -\frac{1+2\eps}{r^2\log r} \quad \text{for $r>R$, and for all $\theta$ and $\Sigma\ni\partial_r$,}
\]
for some $\eps>0$ and $R>1$, or
\item $n\geq 3$, 
\end{enumerate}
we have that the diffusion associated to $\frac{1}{2}\Lap_s$ (from any initial point) is transient.
\end{Cor}

\section{The angular component}

In the final two sections, we consider the asymptotic behavior of $\theta_t=\theta(X_t)$, and in particular, its relationship to the existence of non-constant bounded harmonic functions on minimal submanifolds and sub-Riemannian structures of the above type.

Because the asymptotic behavior of $\theta_t$ is only interesting in the case when $X_t$ is transient, in the remainder of the paper we assume that $X_t$ is transient.

\subsection{The negatively pinched case}

As already mentioned, angular convergence results for martingales of bounded dilation (which include rank-$n$ martingales with $n\geq 2$ as a special case) have been given in the context of the stochastic approach to harmonic maps. In particular, Theorem 2.3.5 of \cite{KendallSurvey}), applied to our context, gives the following.

\begin{THM}\label{THM:Kendall}
If $M$ is a Cartan-Hadamard manifold of dimension $m\geq 3$ and with sectional curvatures pinched between two negative constants, and $X_t$ is a rank-$n$ martingale on $M$ with $2\leq n<m$, then $\theta_t=\theta(X_t)$ converges, almost surely, as $t\rightarrow\zeta$. Further, the distribution of $\theta(\zeta)$ on $\mS^{m-1}$ is ``genuinely random.''
\end{THM}

Let $X_t$ be either Brownian motion on a minimal submanifold or the natural diffusion on a sub-Riemannian structure. Suppose that $E$ is an event in the invariant (or tail) $\sigma$-algebra. Then it is well known that if, for any $x$ on the minimal submanifold or the sub-Riemannian manifold, we let
\[
h(x)= \Prob\lp E\rp \quad\text{given that $X_0=x$,}
\]
then $h$ is a harmonic function, on either the minimal submanifold or the sub-Riemannian manifold (with respect to the sub-La\-plac\-ian), which is clearly bounded. If $E$ is non-trivial, then $h$ will be non-constant. The previous theorem implies that, under these hypotheses, there will be a non-trivial invariant event of the form $\theta_{\zeta}\in U\subset \mS^{m-1}$ for some $U$. If we apply this to the case of sub-Riemannian structures, we immediately get the following.

\begin{Cor}\label{Cor:sRHarmKendall}
Suppose that $M$ is Cartan-Hadamard manifold of dimension $m\geq 3$, and that $M$ has all sectional curvatures bounded above and below by some negative constants. Consider a rank-$n$ sub-Riemannian structure on $M$, with the restriction metric and a volume form such that $\Lap_s$ gives rise to a rank-$n$ martingale. Then we have that $M$ admits a non-constant, bounded, $\Lap_s$-harmonic function.
\end{Cor}

Similarly, applying this result to minimal submanifolds gives Corollary \ref{Cor:UsingKendall}.

\subsection{Convergence of the martingale part}

While the case of negatively pinched curvature is quite natural, one would expect that the upper curvature bound could be improved to allow some decay. Indeed, Goldberg and Mueller \cite{GoldbergMueller} indicate that allowing sub-quadratic curvature decay is possible. In order explore how much further one can go, we start by analyzing the martingale part of the angular process $\theta_t$, in the situation when $n\geq 3$.

We can write the metric on $M$ in polar coordinates around some pole $p$ as follows. Choose a point $\hat{\theta}\in \mS^{m-1}$ and let $U$ be the open hemisphere around $\hat{\theta}$. Then the metric on polar coordinates on $(0,\infty)\times U$ is given by
\[
dr^2 + \sum_{i,j=1}^{m-1} g_{i,j}(r,\theta_1,\ldots,\theta_{m-1}) \, d\theta_i \otimes d\theta_j
\]
where $(\theta_1,\ldots,\theta_{m-1})$ are normal coordinates on $U$ around $\hat{\theta}$ (with respect to the usual metric on $\mS^{m-1}$), and the $g_{i,j}$ are smooth, positive functions on $(0,\infty)\times U$. Of course, this expression is most significant at points of the form $(r,\hat{\theta})=(r,0,\ldots,0)$, since that is where the spherical normal coordinates are centered. At such a point, the matrix $[g_{i,j}]$ gives the inner products of the natural Jacobi fields. To be more precise, let $\gamma$ be the ray (from $p$) through $(r,\hat{\theta})$ and let $j_i(r)$ be the Jacobi field along $\gamma$ determined by the initial conditions $j_i(0)=0$ and $j_i^{\prime}(0)=\partial_{\theta_i}$ (and note that $\{\partial_{\theta_1},\ldots, \partial_{\theta_{m-1}}\}$ gives an orthonormal basis for the subspace of $T_pM$ perpendicular to $\gamma$). Then $\ip{j_i(r)}{j_j(r)}=g_{i,j}(r,0,\ldots,0)$. Thus the Rauch comparison theorem shows that we can bound the the square roots of the eigenvalues of $[g_{i,j}]$ from below by a function $\hat{G}$, and from above by a function $\check{G}$, as in Section \ref{Sect:Compare}.

Now we let $\tilde{w}_{i,t}$ be the component of $v_{i,t}$ orthogonal to $\partial_r$ (that is, the component in the tangent space to $\{r=r_t\}$). Further, suppose, as before, that we rotate our vector fields at some instant so that $v_{2,t}, \ldots, v_{n,t}$ are perpendicular to $\partial_r$, and we let $\vphi_t$ be the angle between $v_{1,t}$ and $\partial_r$, so that $\ip{v_{1,t}}{\partial_r}=\cos\vphi_t$. Then, at this instant, $\tilde{w}_{2,t}, \ldots, \tilde{w}_{n,t}$ all have length 1, while $\tilde{w}_{1,t}$ has length $\sin\vphi_t$. Next, we let $w_{i,t}$ be the image of $\tilde{w}_{i,t}$ on $\mS^{m-1}$ with the usual (round) metric given by the identification of $\mS^{m-1}$ with $\{r=r_t\}$ coming from the polar coordinates. The inner products of the $w_{i,t}$ can be computed from $[g_{i,j}(r_t,\theta_t)]^{-1}$, but we don't need the precise details. The $w_{i,t}$ won't, in general, be orthogonal, but they will be independent (except possibly for $w_{1,t}$ if $\sin\vphi=0$). More to the point, for $\hat{G}$ or $\check{G}$ as above, we see that
\begin{equation}\label{Eqn:4Estimates}
\begin{split}
|w_{1,t}|\leq \frac{|\sin\vphi_t|}{\hat{G}(r_t)}\quad\text{and}\quad 
|w_{i,t}|\leq \frac{1}{\hat{G}(r_t)}\quad\text{for $i\in\{2,\ldots,n\}$} , \\
\text{or}\quad
|w_{1,t}|\geq \frac{|\sin\vphi_t|}{\check{G}(r_t)}\quad\text{and}\quad 
|w_{i,t}|\geq \frac{1}{\check{G}(r_t)}\quad\text{for $i\in\{2,\ldots,n\}$} .
\end{split}
\end{equation}

Next, Ito's rule show the process $\theta_t$ satisfies the SDE
\begin{equation}\label{Eqn:ThetaProcess}
d\theta_t = \sum_{i=1}^n w_{i,t} \, dW^i_t + \lc \text{some vector-valued process} \rc\, dt .
\end{equation}
In other words, we've identified the martingale part of $\theta_t$. In general, the drift term will depend on the derivative of the metric, and we won't have pointwise control of it. (Although see \cite{EltonArticle} for an approach to angular convergence of Brownian motion on a Cartan-Hadamard manifold.) In this section, we focus on just the martingale part.

We're interested in determining conditions under which ``the martingale part converges.'' We put this in quotes because $\theta_t$ takes values in $\mS^{m-1}$, not a vector space, and thus there is no decomposition of $\theta_t$ as a (local) martingale plus a process of (locally) bounded variation (rather, the Ito SDE above gives a ``differential'' version of such a decomposition). Nonetheless, the quadratic variation, which is given by the increasing process 
\begin{equation}\label{Eqn:QVDef}
\QV{\theta}_{t\wedge\zeta} = \int_0^{t\wedge\zeta} \sum_{i=1}^n |w_{i,s}|^2\, ds ,
\end{equation}
remains the object of interest. If we consider the anti-development of $\theta_t$, whether the martingale part converges or not is given by whether the total quadratic variation (that is, $\QV{\theta}_{\zeta}\in (0,\infty]$) is finite. We won't make use of the anti-development here (though see Section 2.3 of \cite{HsuBook}) for details, so we don't explore this further. For our purposes, it's more direct to take the following point of view. Let $y_1,\ldots, y_m$ be standard Euclidean coordinates on $\bR^m$, and thus also functions on $\mS^{m-1}$ via the standard embedding. Then all of the $y_j$ have gradient bounded by 1 on $\mS^{m-1}$. It follows that if $\QV{\theta}_{\zeta}$ is finite, then each $y_j(\theta_t)$ has finite quadratic variation and hence convergent martingale part, and further, each $y_j(\theta_t)$ has quadratic variation bounded above by $\QV{\theta}_{\zeta}$. In the other direction, there is a positive constant $\kappa(m-1)$, depending only on the dimension, such that for any unit vector $v$ at any point of $\mS^{m-1}$, there is a $j\in\{1,\ldots,m\}$ such that $|\ip{v}{\nabla y_j}|\geq \kappa(m-1)$ (where the gradient and inner product are taken on $\mS^{m-1}$ with the standard metric, of course). It follows from this (and the pigeonhole principle) that if $\QV{\theta}_{\zeta}$ is infinite, so is the quadratic variation of at least one of the $y_{j,t}$, and this $y_{j,t}$ thus has non-convergent martingale part. (That a martingale on a Riemannian manifold converges if and only if its quadratic variation is finite is true in general, see \cite{DarlingQuadVar} and \cite{Zheng}.)

Before proving our next theorem, we make one technical observation. Suppose we show that, for any $\beta>0$, there exists $\tilde{B}> 0$ and $\rho>\tilde{B}$ such that
\[
\E\lb  \QV{\theta}_{\sigma_{\tilde{B}}\wedge\zeta} \rb <\beta
\]
whenever $r_0>\rho$. Then this is enough to prove that $\theta_t$ has finite quadratic variation almost surely, in the sense just described. Further, for any $\delta>0$, we can find $\rho$ (perhaps larger than before), such that
\[
\Prob\lp \QV{\theta}_{\zeta} >\delta \rp >1-\delta ,
\]
whenever $r_0>\rho$.
To see this, recall that we're assuming that $X_t$ is transient, and thus for any $\tilde{B}$, we can make $\Prob\lp \sigma_{\tilde{B}}<\infty\rp$ as close to zero as we wish by making $r_0$ large. Since our previous observations apply (up to a set of probability zero) on the set $\{\sigma_{\tilde{B}}=\infty\}$, we see that $\QV{\theta}_{\zeta}$ is almost surely finite on $\{\sigma_{\tilde{B}}=\infty\}$. Further, for any $\delta>0$, we can choose $\rho$ so that
\[
\Prob\lp \sigma_{\tilde{B}}=\infty\text{ and } \QV{\theta}_{\sigma_{\tilde{B}}\wedge\zeta}> \delta \rp >1-\delta
\]
whenever $r_0>\rho$ (here we've used Markov's inequality to pass from a bound on the expectation of $\QV{\theta}_{\cdot}$ to a bound on the probability it exceeds some level). 
So the only issue that remains is seeing that $\theta_t$ also has finite quadratic variation on the set $\{\sigma_{\tilde{B}}<\infty\}$. Again by transience, if $X_t$ hits $\{r=\tilde{B}\}$, then it almost surely hits, say, $\{r=\rho+1\}$ at some later time. At this point, the same result applies, so that $r_t$ stays above $\tilde{B}$ at all future times, and thus $\QV{\theta}_{\cdot}$ stays finite, with probability at least $1-\delta$. Since having finite quadratic variation is a tail-measurable property, by iterating this argument, we see that $\QV{\theta}_{\zeta}$ is almost surely finite, as desired.

\begin{THM}\label{THM:MartCon}
Suppose that $M$ is Cartan-Hadamard manifold of dimension $m\geq 4$. Let $(r,\theta)$ be polar coordinates around some $p$, and let $X_t$ be a rank-$n$ martingale, for $3\leq n< m$. Then if $M$ satisfies the curvature estimate
\[
K(r,\theta,\Sigma) \leq -\frac{\frac{1}{2}+\eps}{r^2\log r} \quad\text{when $r>R$, and for all $\theta$ and $\Sigma\ni\partial_r$,}
\]
for some $\eps>0$ and $R>1$, we have that $\theta_t=\theta(X_t)$ almost surely has finite quadratic variation. Further, for any $0<\delta<1$, there exists $\rho$ (depending only on $M$ and $n$) such that, if $r_0>\rho$, then $\QV{\theta}_{\zeta} <\delta$ with probability at least $1-\delta$.
\end{THM}

\begin{proof} First note that Theorem \ref{THM:Trans} implies that $X_t$ is transient.

Without loss of generality, we can assume that $\eps<1/2$. Then the estimates in Section \ref{Sect:NBig} show that, for some $c>0$ and $B>1$, we can take
\[
\hat{G}(r) = cr\lp \log r\rp^{\frac{1}{2}+\eps} \quad\text{for $r>B$.}
\]
Thus Equation \eqref{Eqn:QVDef} and Inequality \eqref{Eqn:4Estimates} imply that
\[
\E\lb \QV{\theta}_{\sigma_{\tilde{B}}\wedge\zeta}\rb
\leq \frac{n}{c^2} \E\lb \int_0^{\sigma_{\tilde{B}}}\frac{1}{r^2_s\lp\log r_s \rp^{1+2\eps}} \, ds\rb
\]
for any $\tilde{B}>B$. Choose $\beta>0$. In light of the discussion preceding the theorem, it's enough to show that for some $\tilde{B}\geq B$, there exists $\rho>\tilde{B}$ such that
\[
\E\lb \int_0^{\sigma_{\tilde{B}}\wedge\zeta} \frac{1}{r^2_s\lp\log r_s \rp^{1+2\eps}} \, ds\rb <\beta
\]
whenever $r_0>\rho$.

Recall that, in this case,
\[
v_t > \frac{\frac{3}{4}\lp n-1+\sin^2\vphi\rp}{2r_t}\lp 1+\frac{\frac{1}{2}+\eps}{\log r} \rp
\quad\text{for $r_t\geq B$}.
\]
Then for $\alpha>0$, Ito's rule plus the bound on $v_t$ gives (for $r_t>B$)
\[\begin{split}
& \quad\quad d\lp \frac{-1}{\lp\log r_t \rp^{\alpha}} \rp_t = \frac{\alpha\cos\vphi_t}{r\lp\log r\rp^{1+\alpha}}\, dW_t 
+\gamma_t \, dt  \quad \text{where} \\
& \gamma_t \geq \frac{\alpha}{2r_t^2 \lp \log r_t\rp^{1+\alpha}} \\
& \quad\quad\quad \times \lb \frac{3}{4}\lp 1+\frac{\frac{1}{2}+\eps}{\log r_t} \rp \lp n-1+\sin^2\vphi_t \rp-
\cos^2\vphi_t\lp1+\frac{\alpha+1}{\log r_t} \rp \rb .
\end{split}\]
We now take $\alpha<2\eps$. Then using that $n\geq 3$, it's easy to see that we can find $\tilde{B}$ (depending only on $\alpha$ and $\eps$) such that the quantity in brackets in the above expression is at least $1/4$ whenever $r_t>\tilde{B}$. Since we also have $1+\alpha<1+2\eps$, we see that for some $D>0$,
\[
\E\lb \int_0^{\sigma_{\tilde{B}}\wedge\zeta} \frac{1}{r^2_s\lp\log r_s \rp^{1+2\eps}} \, ds\rb <
D \E\lb \int_0^{\sigma_{\tilde{B}}\wedge\zeta} \gamma_s \, ds\rb ,
\]
We further have, using a standard dominated converge argument, that
\[\begin{split}
\E\lb \int_0^{\sigma_{\tilde{B}}\wedge\zeta} \gamma_s \, ds\rb &=
\E \lb \frac{-1}{\lp\log r_{\sigma_{\tilde{B}}\wedge\zeta} \rp^{\alpha}} \rb -  \frac{-1}{\lp\log r_0 \rp^{\alpha}}  \\
& =  \frac{1}{\lp\log r_0 \rp^{\alpha}} - \frac{1}{\lp\log \tilde{B} \rp^{\alpha}}\Prob\lp\sigma_{\tilde{B}} 
<\infty \rp ,
\end{split}\]
where in the last line we've used that $\lim_{t\rightarrow\zeta}r_t=\infty$ to see that $1/(\log r_{\sigma_{\tilde{B}}\wedge\zeta})^{\alpha}$ is zero on the set where $\sigma_{\tilde{B}} =\infty$. Also note that it's the first line where our central idea of controlling an integral along paths by recognizing it as the drift of a semi-martingale (with well-controlled asymptotic behavior) is used. At any rate, the last line of the above is certainly less than $1/(\log r_0)^{\alpha}$ (recall that $\tilde{B}>1$), and that we can make the last line arbitrarily small by making $r_0$ large. Thus it's clear that we can find $\rho>\tilde{B}$ such that
\[
\E\lb \int_0^{\sigma_{\tilde{B}}\wedge\zeta} \frac{1}{r^2_s\lp\log r_s \rp^{1+2\eps}} \, ds\rb <\beta
\]
whenever $r_0>\rho$. This completes the proof.
\end{proof}

To get the complementary result for infinite quadratic variation of $\theta_t$, we again make a preliminary observation.  
Suppose we show that there exists $\tilde{B}> 0$ such that
\[
\QV{\theta}_{\sigma_{\tilde{B}}\wedge\zeta} =\infty \quad\text{almost surely on the set $\{\sigma_{\tilde{B}}=\infty\}$,}
\]
whenever $r_0>\tilde{B}$. Then this is enough to prove that $\theta_t$ has infinite quadratic variation, almost surely. To see this, note that our earlier discussion shows that $\theta_t$ has infinite quadratic variation on $\{\sigma_{\tilde{B}}=\infty\}$ (up to a set of probability zero), and thus we need to consider the set  $\{\sigma_{\tilde{B}}<\infty\}$. The point is that, analogously to the above, transience implies that almost every path that hits $\{r=\tilde{B}\}$ subsequently hits $\{r=r_0\}$. From there, the same result applies, so that paths which don't hit $\{r=\tilde{B}\}$ again almost surely have the property that $\theta_t$ has infinite quadratic variation. Iterating this argument (and noting that almost every path eventually has a last exit time from $\{r\leq \tilde{B}\}$), gives the desired result.

\begin{THM}\label{THM:MartNoCon}
Suppose that $M$ is Cartan-Hadamard manifold of dimension $m\geq 4$. Let $(r,\theta)$ be polar coordinates around some $p$, and let $X_t$ be a rank-$n$ martingale, for $3\leq n< m$. Then if $M$ satisfies the curvature estimate
\[
K(r,\theta,\Sigma) \geq -\frac{1/2}{r^2\log r} \quad\text{when $r>R$, and for all $\theta$ and $\Sigma\ni\partial_r$,}
\]
for some $R>1$,
we have that $\theta_t=\theta(X_t)$ almost surely has infinite quadratic variation.
\end{THM}

\begin{proof} Again, Theorem \ref{THM:Trans} implies that $X_t$ is transient.

Using the estimates of Section \ref{Sect:NBig}, we have that, for some $c>0$ and $B>R$, we can take
\[\begin{split}
& \check{G}(r) \leq c r\lp \log r \rp^{\frac{1}{2}}\lp\ltwo r \rp^{\frac{1}{2}} \quad\text{for $r>B$, and} \\
& \frac{3}{4}  \frac{n-1+\sin^2\vphi_t}{2r_t}\leq v_t\leq \frac{5}{4}  \frac{n-1+\sin^2\vphi_t}{2r_t}
\quad\text{for $r_t>B$.}
\end{split}\]
Hence, Equation \eqref{Eqn:QVDef} and Inequality \eqref{Eqn:4Estimates} imply that
\[
\QV{\theta}_{\sigma_{\tilde{B}}\wedge\zeta}
\geq \frac{n-1}{c^2} \int_0^{\sigma_{\tilde{B}}\wedge\zeta} \frac{1}{r_s^2 \lp\log r_s\rp \lp\ltwo r_s\rp}\, ds
\]
Thus, in light of the discussion before the theorem, it's enough for us to show that for some $\tilde{B}> B$,
\[
\int_0^{\sigma_{\tilde{B}}\wedge\zeta} \frac{1}{r_s^2 \lp\log r_s\rp \lp\ltwo r_s\rp}\, ds = \infty \quad\text{almost surely on the set $\{\sigma_{\tilde{B}}=\infty\}$,}
\]
whenever $r_0>\tilde{B}$.

Note that
\[\begin{split}
\lp \lthree\rp^{\prime}(r) &= \frac{1}{r \lp\log r\rp\lp \ltwo r\rp} \quad\text{and}\\
\lp \lthree\rp^{\prime\prime}(r) &=  \frac{-1}{r^2 \lp\log r\rp\lp \ltwo r\rp}\lp 1+\frac{1}{\log r}+\frac{1}{\lp\log r\rp\lp \ltwo r\rp} \rp .
\end{split}\]
Then Ito's rule plus the two-sided bound on $v_t$ lets us compute that, for $r_t>B$,
\[\begin{split}
& \quad d\lp \lthree\rp_t r = \frac{\cos\vphi_t}{r_t \lp\log r_t\rp\lp \ltwo r_t\rp} \, dW_t +\gamma_t \, dt \quad\text{where} \\
& \frac{ \frac{3}{4}\lp n-1+\sin^2\vphi_t\rp -\cos^2\vphi_t\lp 1+\frac{1}{\log r_t}+\frac{1}{\lp\log r_t\rp\lp \ltwo r_t\rp}\rp}{2r^2_t \lp\log r_t\rp\lp \ltwo r_t\rp} 
\leq \gamma_t \\ 
& \leq 
\frac{ \frac{5}{4}\lp n-1+\sin^2\vphi_t\rp -\cos^2\vphi_t\lp 1+\frac{1}{\log r_t}+\frac{1}{\lp\log r_t\rp\lp \ltwo r_t\rp}\rp}{2r^2_t \lp\log r_t\rp\lp \ltwo r_t\rp} 
\end{split}\]
Observe that we can find $\tilde{B}>B$ such that the lower bound on $\gamma_t$ is positive when $r_t>\tilde{B}$. Thus $\lthree r_t$ is a (local) sub-martingale (on the set where $r>\tilde{B}$). Further, the upper bound implies that, perhaps after increasing $\tilde{B}$, we can find $D>0$ such that
\[
\frac{1}{D} \int_0^{\sigma_{\tilde{B}}\wedge\zeta} \gamma_s \, ds \leq  
\int_0^{\sigma_{\tilde{B}}\wedge\zeta} \frac{1}{r_s^2 \lp\log r_s\rp \lp\ltwo r_s\rp}\, ds
\]
whenever $r_0>\tilde{B}$.

By the transience of $r_t$, it follows that $\lthree r_{\sigma_{\tilde{B}}\wedge\zeta}=\infty$ almost surely on the set $\{\sigma_{\tilde{B}}=\infty\}$ (we assume that $r_0>\tilde{B}$). Now a (local) submartingale can diverge to $\infty$ only if its bounded variation part diverges to infinity (up to a set of probability zero). This follows from the fact that the martingale part either converges (if the quadratic variation remains bounded) or hits every real value infinitely often (if the quadratic variation increases without bound), up to a set of probability zero. Thus,
\[
\int_0^{\sigma_{\tilde{B}}\wedge\zeta} \gamma_s \, ds = \infty \quad\text{almost surely on the set $\{\sigma_{\tilde{B}}=\infty\}$,}
\]
whenever $r_0>\tilde{B}$. Since $1/D$ is just a positive constant, we've succeeded in showing that
\[
\int_0^{\sigma_{\tilde{B}}\wedge\zeta} \frac{1}{r_s^2 \lp\log r_s\rp \lp\ltwo r_s\rp}\, ds = \infty \quad\text{almost surely on the set $\{\sigma_{\tilde{B}}=\infty\}$,}
\]
whenever $r_0>\tilde{B}$. This completes the proof.
\end{proof}

\section{Angular convergence in the radially symmetric case}

As noted, while we have good control over the convergence or non-con\-ver\-gence of the martingale part of the angular process $\theta_t$, in general the drift term will not be so easy to handle. If we restrict our attention to the radially symmetric case, the expression for the drift becomes more manageable. Indeed, for the remainder of the paper, we assume the $M$ is radially symmetric and determine conditions under which the angular process converges.

\subsection{Basic computations}

If $M$ is radially symmetric around a point $p$, we can write the metric on $M$ in polar coordinates around $p$ as
\[
dr^2 + G^2(r) d\theta^2 \quad\text{where $d\theta^2$ is the standard metric on $\mS^{m-1}$,} 
\]
for a (smooth) function $G:[0,\infty)\rightarrow[0,\infty)$ with $G(0)=0$ and $G^{\prime}(0)=1$. Our notation is consistent in that $G$ gives the length of the natural Jacobi fields, is a solution to the (scalar) Jacobi equation for the radially symmetric sectional curvature function $K(r)=K(r,\theta,\Sigma)$ for all $\theta$ and $\Sigma\ni\partial_r$, and is bounded above and below by the appropriate comparison functions $\check{G}$ and $\hat{G}$, respectively. 

If we let $w_{i,t}$ be as above, we note that now the $w_{i,t}$ are orthogonal (except for $w_{1,t}$ if $\sin\vphi_t=0$), and we have
\[
|w_{1,t}|=\frac{|\sin\vphi_t|}{G(r_t)}\quad \text{and} \quad |w_{2,t}|=\cdots=|w_{n,t}|=\frac{1}{G(r_t)} .
\]
Of course, Inequality \eqref{Eqn:4Estimates} still holds.

The next step in computing the drift of $\theta_t$ is to compute the Hessian of $\theta$, in an appropriate sense. If we choose some $\hat{\theta}\in \mS^{m-1}$ and again let $(\theta_1,\ldots,\theta_{m-1})$ be normal coordinates on $\mS^{m-1}$ around $\hat{\theta}$, then at any point $(r,\hat{\theta})$ (with $r>0$), the metric is given in these coordinates by the diagonal matrix
\[
\begin{bmatrix}
1 &  & \,\,\,\,\,\, \bigzero &  \\
   & G^2(r) & & \\
  & \!\!\!\!\!\! \bigzero &  \ddots & \\
 & & & G^2(r)
\end{bmatrix}
\]
up to first order. In particular, this is sufficient to compute the Christoffel symbols at $(r,\hat{\theta})$, and thus also the Hessian of the coordinate functions at this point (with respect to these coordinates).

If we use the sub/superscript $r$ to denote the $r$ coordinate and $i$ to denote the $\theta_i$ coordinate, then we have that (at the point $(r,\hat{\theta})$ with $r>0$)
\[
\Gamma^r_{ii}= \lp G\cdot G^{\prime}\rp(r) \quad\text{and}\quad
\Gamma^i_{ir}=\Gamma^i_{ri}= \lp \frac{G^{\prime}}{G}\rp(r) ,
\]
and all other Christoffel symbols are zero. From here we compute that the Hessian of $\theta_i$, as a bilinear form, is given by 
\[
- \lp \frac{G^{\prime}}{G}\rp(r) \lp dr \otimes d\theta_i +  d\theta_i \otimes dr\rp .
\]

Assume that the process is at $(r,\hat{\theta})$). After potentially rotating our normal coordinates, we can assume that $w_{1,t}=\sin\vphi_t\partial_{\theta_1}$. Since $|\partial_{\theta_1}|=1/G(r)$ here, at this instant we see that the martingale part of $r_t$ is generated by $\cos\vphi_t \, dW^1_t$ and the martingale part of $\theta_{1,t}$ is generated by $(\sin\vphi_t/G(r_t)) \, dW^1_t$.  Then Ito's rule implies that, at this instant, the drift of $\theta_{1,t}$ is given by
\[
-\frac{G^{\prime}(r_t)}{G^2(r_t)} \sin\vphi_t \cos\vphi_t \, dt .
\]

Globally, we see that $\theta_t$ satisfies the SDE
\begin{equation}\label{Eqn:ThetaSDE} 
d\theta_t = \sum_{i=1}^n w_{i,t} \, dW^i_t
-\frac{G^{\prime}(r_t)}{G(r_t)}\cos\vphi_t w_{1,t}\, dt .
\end{equation}
Thus the total quadratic variation of $\theta_t$ is given by
\[
\QV{\theta}_{\zeta} = \int_0^{\zeta} \frac{n-1+\sin^2\vphi_t}{G^2(r_t)} \, dt ,
\]
although we won't explicitly use this, since the results of the previous section are sufficient here too. More important, the total drift (or total variation of the bounded variation part) of $\theta_t$ is given by
\[
\int_0^{\zeta} \lab \sin\vphi_t \cos\vphi_t\rab \frac{G^{\prime}(r_t)}{G^2(r_t)} \, dt ,
\]
which may be infinite.

We see that, in contrast to the situation for the quadratic variation, whether or not the total variation of the drift is finite or not depends in general on the behavior of $\vphi_t$. Two special cases in which $\theta_t$ is a martingale are worth mentioning.

\subsection{Vanishing angular drift}\label{Sect:VanishingDrift}

First, if $\varphi_t\equiv 0$, $X_t$ is restricted to a totally geodesic submanifold $N$ of dimension $n$ through $p$. In particular, we can find normal coordinates $(y_1,\ldots,y_m)$ around $p$ such that the $N$ is the set $\{y_{n+1}=\cdots=y_m=0\}$, and $N$ is simply a rotationally symmetric manifold of dimension $n$ with the metric given in polar coordinates by the same function $G(r)$ as for $M$. Then $X_t$ is just Brownian motion on $N$. (Thus we see that Brownian motion on a radially symmetric Cartan-Hadamard manifold is actually a special case of the above, in spite of our taking $n<m$ in our definition of a rank-$n$ martingale.) The drift term in Equation \eqref{Eqn:ThetaSDE} vanishes identically), so that the angular convergence or non-convergence of $X_t$ is given exactly by the conditions in Theorems \ref{THM:MartCon} and \ref{THM:MartNoCon}. In light of this, it is unsurprising that the conditions on $G$ in Theorems \ref{THM:MartCon} and \ref{THM:MartNoCon} are the same as those found by March \cite{March} for Brownian motion on radially symmetric Cartan-Hadamard manifolds. However, the approach used to prove both of these theorems differs from the method used to by March (and also in \cite{HsuBook}), which transforms the question of angular convergence to the question of whether or not a certain one-dimensional diffusion, coming from a time-change of the radial process, has finite lifetime (for which an answer is essentially known). In the case of rank-$n$ martingales, the appearance of $\vphi_t$ (which functions here mostly as a nuisance parameter) prevents us from being able to derive a simple one-dimensional diffusion, and it's not immediately clear how to adapt this method. Our approach gets around this, and also has the advantage of (arguably) being somewhat more elementary. (On the other hand, this earlier work is able to also get sharp curvature bounds in the $n=2$ case, which we don't address here.)

The second case is when $\vphi\equiv \pi/2$. Then again the drift term in Equation \eqref{Eqn:ThetaSDE} vanishes identically, and further, the evolution of $r_t$ is deterministic. That is, Equation \eqref{Eqn:BasicDecomp} reduces to an ODE (independent of $\theta$), namely
\[
dr_t = \frac{n}{2}\frac{G^{\prime}(r_t)}{G(r_t)} \, dt .
\]
Thus $X_t$ is supported on the (time-varying) sphere of radius $r(t)$ that comes from solving this ODE with the initial condition $r(0)=r_0>0$ (where we assume $X_t$ doesn't start at $p$ to avoid degeneracy in the polar coordinates and ensure that $\Lambda_t$ is continuous). Further, in the case when $n=m-1$, we must have that $\Lambda_t$ is the tangent plane to the sphere of radius $R(t)$ around $p$ (at $\theta_t$). Then we see that $X_t$ is the inhomogeneous diffusion along the backward mean curvature flow starting from the sphere of radius $r_0$ around $p$ (and where the diffusion is started at $(r_0,\theta_0)$. Further, $\theta_t$ is a time-changed Brownian motion on $\mS^{m-1}$, where the time-change factor is a function of the radius $r(t)$. Finally, $\theta_t$ converges or not (corresponding to whether or not the changed time remans bounded) according to Theorems \ref{THM:MartCon} and \ref{THM:MartNoCon}.

\subsection{Results for rank-$n$ martingales}

When the drift does not vanish, the situation is more complicated. Not only does the finiteness of the total variation of the drift depend on $\vphi_t$, as already mentioned, but (focusing our attention just on the drift) a path on the sphere of locally bounded variation that accumulates infinite total variation over its lifetime may or may not converge, depending on the particular ``cancellations'' that occur. The upshot of these considerations is that we cannot give as complete a description of the behavior of $\theta_t$ as we can for its martingale part alone.

We are able to give curvature bounds under which $\theta_t$ must converge, regardless of the behavior of $\vphi_t$, and doing so is the goal of this section. This is clearly a result of the same type as Theorem \ref{THM:Kendall}. The difference is that we allow some quadratic decay of the upper curvature bound, which goes beyond the sub-quadratic decay mentioned by Goldberg and Mueller \cite{GoldbergMueller}. On the other hand, we only consider the radially symmetric case (and where $n\geq 3$). It would not be surprising if a similar result held in more generality, but restricting ourselves to the rotationally symmetric case also allows us to continue to use the sort of elementary stochastic techniques we have employed through the paper.

We require a preliminary lemma.

\begin{Lemma}\label{Lem:Halfspace}
Let $M$ be an $m$-dimensional Cartan-Hadamard manifold that is rotationally symmetric about some point $p$, let $r$ be the distance from $p$ (as usual), and let $(x_1,\ldots,x_m)$ be any set of normal coordinates around $p$. If $X_t$ is a rank-$n$ martingale, for $2\leq n<m$, with an initial distribution (not necessarily a point mass) such that $\Prob\lp x_1(X_0)>0 \text{ and }r(X_0)<A \rp>0$, for some $A>0$, then for any $B>A$
\[
\Prob\lp x_1\lp X_{\sigma_B}\rp >0 \text{ and } \sigma_{B}<\infty \rp >0, 
\]
where $\sigma_B$ is the first hitting time of the set $\{r=B\}$.
\end{Lemma}
\begin{proof}
By radial symmetry, the hypersurface $\{x_1=0\}$ is totally geodesic. Let $\rho$ be the signed distance from $\{x_1=0\}$, with $\rho>0$ corresponding to $x_1>0$. Because all the sectional curvatures are non-negative, on the set $\{\rho>0\}$, the Hessian of $\rho$ is bounded from below by 0 (which is the analogous quantity for the comparison manifold, Euclidean space). In particular (writing $\rho_t=\rho(X_t)$ as usual), if $\rho_0>0$ then $\rho_t$ is a sub-martingale until the first hitting time of $\{\rho=0\}$. Denote this hitting time by $\eta$. Then if $\eta$ is almost surely finite, $\E\lb\QV{\rho}_{\eta}\rb$ must be infinite, simply by comparison with a (time-changed) one-dimensional Brownian motion.

Now by conditioning on the event $\{x_1(X_0)>0 \text{ and }r(X_0)<A\}$, we can assume that this holds almost surely, without loss of generality. Then we know from Lemma \ref{BasicLemma} that $\sigma_B<\infty$ almost surely. Moreover, we showed, in the proof of Lemma \ref{BasicLemma}, that the expectation of $\sigma_B$ is finite (in particular, less than $B^2/n$). Because $\rho$ is a distance function, $\QV{\rho}_{t\wedge\sigma_B} = t\wedge\sigma_B$, and thus $\E\lb \QV{\rho}_{\sigma_B}\rb<\infty$. Comparing this to the results for $\eta$ above, it follows that $\eta>\sigma_B$ with positive probability. Since $\eta>\sigma_B$ implies that $x_1(X_{\sigma_B})>0$ (just by the definitions of $\rho$ and $\eta$), the lemma is now proved.
\end{proof}

Note that, applied to minimal submanifolds, this lemma gives a type of maximum principle, relative to minimal hypersurfaces of the form $\{y_1=0\}$.

\begin{THM}\label{THM:AngleCon}
Suppose that $M$ is Cartan-Hadamard manifold of dimension $m\geq 4$, and that $M$ is radially symmetric around some point $p$. Let $(r,\theta)$ be polar coordinates around $p$, and let $X_t$ be a rank-$n$ martingale, for $3\leq n< m$. Further, assume that $M$ satisfies the curvature estimate
\[
-a^2 \leq
K(r,\theta,\Sigma) \leq -\frac{2+\eps}{r^2} \quad\text{when $r>R$, and for all $\theta$ and $\Sigma\ni\partial_r$,}
\]
for some $a>0$, $\eps>0$, and $R>1$.
Then we have that $\theta_t=\theta(X_t)$ converges, almost surely, as $t\rightarrow\zeta$, and this limit $\theta_{\zeta}$ is not a point mass. Further, for any $0<\delta<1$, there exists $\rho$ (depending only on $M$ and $n$) such that, if $r_0>\rho$, then $\theta_{\zeta}\in B_{\delta}(\theta_0) \subset\mS^{m-1}$ with probability at least $1-\delta$.
\end{THM}

\begin{proof} By Theorem \ref{THM:MartCon} and the upper curvature bound, we know that the martingale part of $\theta_t$ almost surely has finite quadratic variation and that this quadratic variation can be made less than $\delta_0\in(0,1)$ with probability at least $1-\delta_0$ by taking $r_0$ large enough. We also see that $X_t$ is transient. Next, we wish to establish the analogous result for the the total variation of the drift of $\theta_t$. In light of our earlier discussion of the drift, we see that the total variation is pathwise bounded from above by
\[
\int_0^{\zeta} \frac{G^{\prime}}{G^2}(r_s)\, ds ,
\]
and we recall that the integrand is always positive.

The constant lower curvature bound and the results of Section \ref{Sect:ConstEst} imply that, for some $C>0$ depending only on $a$ and some $B>R$, we have
\[
\frac{G^{\prime}}{G}(r) \leq C \quad\text{whenever $r>B$.} 
\]
(That $G^{\prime}/G$ can be estimated from above by $\check{G}^{\prime}/\check{G}$ is just the radially symmetric version of the Hessian comparison theorem, see Lemma 6.4.3 of \cite{HsuBook}.) The upper curvature bound implies that, after possibly increasing $B$, for some $c>0$, we have both
\[
\frac{1}{G}(r) < \frac{c}{r^{2+\delta_1}} \quad\text{for $r>B$}
\text{ and}\quad v_t> \frac{n-1+\sin^2 \varphi_t}{r_t} \quad\text{for $r_t>B$} ,
\]
where $\delta_1$ depends on $\eps$ as in Section \ref{Sect:NBig}. One consequence is that
\[
\frac{G^{\prime}}{G^2}(r_t) < \frac{C c}{r^{2+\delta_1}} \quad\text{for $r>B$} .
\]
Recall also that $G^{\prime}>0$, and thus $G^{\prime}/G^2$ is always positive.

Further, Ito's rule gives, for $r_t>B$,
\[\begin{split}
& \quad d\lp \frac{-1}{\log r}\rp_t = 
\frac{\cos\vphi_t}{r\lp\log r\rp^{2}}\, dW_t 
+\gamma_t \, dt  \quad \text{where} \\
& \gamma_t= \frac{v_t}{r\lp\log r\rp^{2}} - \frac{\cos^2\varphi_t}{2r^2\lp\log r\rp^{2}}\lb 1+\frac{2}{\log r}\rb .
\end{split}\]
Then the bound on $v_t$ and the assumption that $n\geq 3$ imply that, after possibly increasing $B$,
\[\begin{split}
\gamma_t &\geq  \frac{1}{r^2\lp\log r\rp^{2}}\lc n-1+\sin^2\varphi_t-\frac{\cos^2\varphi_t}{2}
\lb 1+\frac{2}{\log r}\rb\rc \\
&> \frac{1}{r^2\lp\log r\rp^{2}} \quad \text{for $r>B$} .
\end{split}\]
Thus, after possibly increasing $B$ again, we have 
\[
\gamma_t > \frac{C c}{r^{2+\delta_1}} \quad\text{for $r>B$} .
\]

Then, just as in the proof of Theorem \ref{THM:MartCon}, we have, for any $r_0>B$,
\[\begin{split}
\E\lb \int_0^{\sigma_{B}\wedge\zeta} \frac{G^{\prime}}{G^2}(r_s) \, ds\rb &<
\E\lb \int_0^{\sigma_{B}\wedge\zeta} \gamma_s \, ds\rb \\
& =  \frac{1}{\log r_0} - \frac{1}{\log B}\Prob\lp\sigma_{B} <\infty \rp .
\end{split}\]
This last line can be made arbitrarily close to zero by taking $r_0$ large. Thus (again as in the proof of Theorem \ref{THM:MartCon}) the total variation of the drift of $\theta_t$ is almost surely finite and, for any $\delta_0\in (0,1)$, we can find $\rho>B$ (depending only on $M$ and $n$) such that if $r_0>\rho$, the total variation of the drift of $\theta_t$ is less than $\delta_0$ with probability at least $1-\delta_0$.

Next, we wish to see that $\theta_t$ converges almost surely.  If we again let $y_1,\ldots, y_m$ be standard Euclidean coordinates on $\bR^m$, and thus also functions on $\mS^{m-1}$ via the standard embedding, all of the $y_i$ have bounded gradient and bounded Hessian on $\mS^{m-1}$. It follows that $y_i(\theta_t)$ has finite quadratic variation and finite drift (for each $i$), and thus each $y_i(\theta_t)$ converges almost surely (as $t\rightarrow \zeta$). So $\theta_t$ converges almost surely to some $\theta_{\zeta}\in\mS^{m-1}$. Similarly, for any $\tilde{\delta}\in(0,1)$, it's clear that if the quadratic variation and total variation of the drift of $\theta_t$ are both small enough with high enough probability, then $y_i(\theta_t)$ stays within distance $\tilde{\delta}$ of $y_i(\theta_0)$ with probability at least $1-\tilde{\delta}$. Then it follows from our earlier estimates that for any $0<\delta<1$, there exists $\rho$ (depending only on $M$ and $n$) such that, if $r_0>\rho$, then $\theta_{\zeta}\in B_{\delta}(\theta_0)$ with probability at least $1-\delta$, as desired.

Finally, to prove that $\theta_{\zeta}$ is not a point mass, we proceed by contradiction. Namely, assume that $\theta_{\zeta}=\hat{\theta}$, for some $\hat{\theta}\in\mS^{m-1}$, almost surely. At any finite time $\tau>0$, Equation \eqref{Eqn:ThetaSDE} implies that $\theta_{\tau}$ cannot be a point mass. Thus, we can find normal coordinates $(x_1,\ldots,x_m)$ around $p$ such that $x_1(r,\hat{\theta})<0$ whenever $r>0$ and $x_1(X_{\tau})>0$ with positive probability. Note that $\{x_1=0\}$ corresponds to a great circle in the sphere at infinity, and thus the distance, in $\mS^{m-1}$, between $\hat{\theta}$ and $\{x_1=0\}$ is positive (and $\hat{\theta}$ is in the ``negative $x_1$'' hemisphere). Denote this distance by $d$. By the previous parts of the theorem, there exists $A>0$ such that, if $r_t\geq A$, then $\theta_{\zeta}\in B_{d/2}(\theta_t)$ with probability at least $1/2$.

After possibly increasing $A$, we know that we have $x_1(X_{\tau})>0$ and $r(X_{\tau})<A$ with positive probability. Thus, Lemma \ref{Lem:Halfspace} implies that $x_1(X_{\sigma_{A}})>0$ with positive probability; let the set of such paths be denoted by $S$. Observe that for paths in $S$, the $\mS^{m-1}$-distance between $\theta_{\sigma_A}$ and $\hat{\theta}$ is at least $d$. Further, by our choice of $A$, $\Prob\lp \theta_{\zeta}\in B_{d/2}(\theta_{\sigma_A}) | S\rp \geq 1/2$, and then by our choice of $d$ and $S$, $\Prob\lp \theta_{\zeta}\not\in B_{d/2}(\hat{\theta}) | S\rp \geq 1/2$. Since $S$ has positive probability, this contradicts our assumption that $\theta_{\zeta}=\hat{\theta}$ almost surely, and this contradiction finishes the proof.

\end{proof}

\subsection{Geometric consequences}

The geometric implications of Theorem \ref{THM:AngleCon} are not obvious in every context (such as for ancient solutions of the mean curvature flow). Nonetheless, there are some things we can say.

First, again consider a rank-$n$ sub-Riemannian structure on $M$, with the restriction metric and a volume form such that $\Lap_s$ gives rise to a rank-$n$ martingale (as discussed in Section \ref{Sect:SubR}). Then let $X_t$ be the associated diffusion. In this case, $X_t$ is Markov (indeed, the $v_{i,t}$ in Equation \eqref{Eqn:RankN} can and should be chosen to be locally smooth, so that the equation has a unique solution by standard results for SDEs) and has a positive density  with respect to the volume form at any positive time (by a famous result of H\"ormander and later via the Malliavin calculus). Assume that $M$ (and the sub-Riemannian structure) satisfy the hypotheses of Theorem \ref{THM:AngleCon}. We now let $U$ be an open, nonempty subset of $\mS^{m-1}$ such that the complement $U^c$ has non-empty interior, and let $p(y)$ be the probability that the diffusion, started from any $y=X_0\in M$, has $\theta_{\zeta}\in U$. Then $p$ is a non-constant, bounded $\Lap_s$-harmonic function on $M$. (The $\Lap_s$-harmonicity is a consequence of the facts that $X_t$ is Markov and the event $\theta_{\zeta}\in U$ is tail-measurable.) Indeed, for $\tilde{\theta}\in U$, we see that $p(\tilde{r},\tilde{\theta})\rightarrow 1$ as $\tilde{r}\rightarrow \infty$, and similarly, for $\tilde{\theta}$  in the interior of $U^c$, $p(\tilde{r},\tilde{\theta})\rightarrow 0$ as $\tilde{r}\rightarrow \infty$. These observations form the basis for a probabilistic approach to a version of the Dirichlet problem at infinity, relative to the sphere at infinity determined by the original Riemannian structure on $M$ (see Proposition 6.1.1 of \cite{HsuBook}). However, we don't pursue this any further (as already mentioned) and instead simply give the following corollary of Theorem \ref{THM:AngleCon}.

\begin{Cor}\label{Cor:sRHarm}
Suppose that $M$ is Cartan-Hadamard manifold of dimension $m\geq 4$, and that $M$ is radially symmetric around some point $p$. Let $(r,\theta)$ be polar coordinates around $p$.  With $3\leq n<m$, consider a rank-$n$ sub-Riemannian structure on $M$, with the restriction metric and a volume form such that $\Lap_s$  gives rise to a rank-$n$ martingale. Then if $M$ satisfies the curvature estimate
\[
K(r,\theta,\Sigma) \leq -\frac{\frac{1}{2}+\eps}{r^2\log r} \quad\text{when $r>R$, and for all $\theta$ and $\Sigma\ni\partial_r$,}
\]
for some $\eps>0$ and $R>1$, we have that $M$ admits a non-constant, bounded, $\Lap_s$-harmonic function.
\end{Cor}

Similar logic shows that, under the same hypotheses on curvature, an $n$-dim\-en\-sion\-al minimal submanifold $N$ will admit a non-constant, bounded, harmonic function (where the harmonicity is with respect to the induced Laplacian on $N$, of course). Here, the relationship between harmonic functions and the sphere at infinity (of $M$) is less straightforward, since the accumulation points of $N$ in this sphere at infinity will generally have a more complicated structure. However, in order to establish the existence of a non-constant, bounded, harmonic function, we need only show there is some set $U$ as above such that we can always find points with $\theta$-coordinates in each of $U$ and the interior of $U^c$ for arbitrarily large $r$-coordinates. The only way we could fail to be able to find a set $U$ as desired is if there were some $\hat{\theta}$ such that for all sequences $x_i \in N$ with $r(x_i)\rightarrow \infty$, we had $\theta(x_i)\rightarrow \tilde{\theta}$. However, this would imply that $\theta_{\zeta}= \hat{\theta}$ for every path (since $N$ is transient), contradicting the fact that $\theta_{\zeta}$ is not a point mass. With this in mind, we have proven Corollary \ref{Cor:MinHarm}.

\def\cprime{$'$}

\end{document}